\documentclass[12pt, a4paper, openright]{scrbook}

\usepackage{latexsym,amssymb,mathrsfs,amsthm,amsmath}
\usepackage{graphicx}
\usepackage[latin1]{inputenc}
\usepackage{natbib}
\usepackage{varioref}

\newcommand{\beqq}{\begin{equation}}
\newcommand{\eeqq}{\end{equation}}
\newcommand{\beao}{\begin{eqnarray*}}
\newcommand{\eeao}{\end{eqnarray*}\noindent}
\newcommand{\beam}{\begin{eqnarray}}
\newcommand{\eeam}{\end{eqnarray}\noindent}

\newcommand{\bbr}{\mathbb{R}}

\newcommand{\bbn}{\mathbb{N}}

\newcommand{\eqd}{\stackrel{\mathscr{D}}{=}}
\newcommand{\limw}{\stackrel{weak}{\rightarrow}}

\newcommand{\limas}{\stackrel{\rm a.s.}{\rightarrow}}

\newcommand{\nto}{{n\to\infty}}
\newcommand{\tto}{{t\to\infty}}
\newcommand{\limn}{\stackrel{\nto}{\rightarrow}}

\newcommand{\cala}{{\mathcal{A}}}
\newcommand{\calb}{{\mathcal{B}}}
\newcommand{\calc}{{\mathcal{C}}}

\newcommand{\calf}{\mathcal{F}}
\newcommand{\calg}{\mathcal{G}}
\newcommand{\calm}{\mathcal{M}}
\newcommand{\caln}{{\mathcal{N}}}
\newcommand{\cals}{{\mathcal{S}}}
\newcommand{\calu}{{\mathcal{U}}}
\newcommand{\calw}{\mathcal{W}}
\newcommand{\calz}{{\mathcal{Z}}}
\newcommand{\calbm}{\mathcal{BM}}
\newcommand{\calwp}{\mathcal{WP}}
\newcommand{\caloup}{\mathcal{OUP}}

\newcommand{\al}{{\alpha}}

\newcommand{\cov}{{\mathrm{cov}}}

\newcommand{\rmt}{\mathrm{T}}
\newcommand{\etr}{\mathrm{etr}}
\newcommand{\tr}{\mathrm{tr}}
\newcommand{\diag}{\mathrm{diag}}
\newcommand{\rmbesq}{\mathrm{BESQ}}

\newcommand{\ov}{\overline}
\newcommand{\wh}{\widehat}
\newcommand{\wt}{\widetilde}

\newcommand{\fa}{\forall\,}

\newtheorem{Theorem}{Theorem}[chapter]

\newtheorem{Corollary}[Theorem]{Corollary}

\newtheorem{Lemma}[Theorem]{Lemma}
\newtheorem{Definition}[Theorem]{Definition}
\newtheorem{Example}{Example}

\newtheorem{Remark}[Theorem]{Remark}
\newtheorem{Assumption}[Theorem]{Assumption}

\begin{document}

\begin{titlepage}
\vspace*{2cm}
\pagestyle{empty}
\begin{center}
 \bf\huge Wishart Processes 
\end{center}
\vspace{3cm}
\begin{center}
 \LARGE OLIVER PFAFFEL
\end{center}
\vspace{3cm}
\begin{center}
\LARGE Student research project \\
supervised by Claudia Klüppelberg and Robert Stelzer
\end{center}
\vspace{1.5cm}
\begin{center}
 \LARGE September 8, 2008
\end{center}
\vspace{3cm}
\begin{center}
\includegraphics{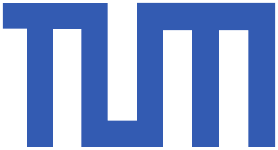}
\end{center}
\begin{center}
\LARGE{Technische Universit\"at M\"unchen\\
Fakult\"at f\"ur Mathematik}
\end{center}
\end{titlepage}

%
%
%

\cleardoublepage

\pagenumbering{roman} 
\pagestyle{headings} 

\tableofcontents

\chapter{Introduction}  \pagenumbering{arabic}\setcounter{page}{1} 

In this thesis we consider a matrix variate extension of the Cox-Ingersoll-Ross process (see \citet{cox}), i.e. a solution of a one-dimensional stochastic differential equation of the form
\beqq
dS_t=\sigma\sqrt{S_t}\,dB_t+a(b-S_t)\,dt, \quad S_0=s_0\in\bbr
\label{intro_cir}\eeqq
with positive numbers $a,b,\sigma$ and a one-dimensional Brownian motion $B$. Our extension is defined by a solution of the $p\times p$-dimensional stochastic differential equation of the form
\beqq
dS_t=\sqrt{S_t}\,dB_t\,Q+Q^\rmt dB_t^\rmt \sqrt{S_t}+(S_t K+K^\rmt S_t+ \al Q^\rmt Q)\,dt, \quad S_0=s_0\in\calm_p(\bbr)
\label{intro_wp}\eeqq
where $Q$ and $K$ are real valued $p\times p$-matrices, $\al$ a non-negative number and $B$ a $p\times p$-dimensional Brownian motion (that is, a matrix of $p^2$ independent one-dimensional Brownian motions).\\
While it is well-known that solutions of (\ref{intro_cir}), called CIR processes, always exist, the situation for (\ref{intro_wp}) is more difficult. As we will see in this thesis, it is crucial to choose the parameter $\al$ in the right way, to guarantee the (strong) existence of unique solutions of (\ref{intro_wp}), that are then called Wishart processes. We will derive that it is sufficient to choose the parameter $\alpha$ larger or equal to $p+1$. That is similar to a result given by \citet{bru}.

The characteristic fact of (\ref{intro_cir}) is that this process remains positive for a certain choice of $b$. This makes it suitable for modeling for example an interest rate, which should always be positive because of economic reasons. Hence, this is an approach for pricing bonds (see the section 5.1). If we want to consider some (corporate) bonds jointly, e.g. because they are correlated, the need for a multidimensional extension comes up. See section 5.3 of this thesis or \citet[3.5.2.]{gou} for a discussion of this topic. For the Wishart processes, we have in the case $\al\geq p+1$ that the unique solution of (\ref{intro_wp}) remains positive definite for all times.\\
Another well-known fact is that the conditional(on $s_0$) distribution of the CIR process at a certain point in time is noncentral chi-square. We will see that the conditional distribution of the Wishart process $S_t$ at time $t$ is a matrix variate extension of the noncentral chi-square distribution, that is called noncentral Wishart distribution.
\vspace{0.5cm}\\
An application for matrix variate stochastic processes can be found in \citet{fonseca}, which model the dynamics of a $p$ risky assets $X$ by
\beqq dX_t=\diag(X_t)[(r\mathbf{1}+\lambda_t)\,dt+\sqrt{S_t}\,dW_t] \label{eq_fonseca}\eeqq
where $r$ is a positive number, $\mathbf{1}=(1,\ldots,1)\in\bbr^p$, $\lambda_t$ a $p$-dimensional stochastic process, interpreted as the risk premium, and $Z$ a $p$-dimensional Brownian motion. The volatility process $S$ is the Wishart process of (\ref{intro_wp}).\\

In chapter 2, we introduce some notations and give the necessary background for understanding the following chapters. In chapter 3, a review of fundamental terms and results of matrix variate stochastics and the theory of stochastic differential equations is given, and in section 3.5 some results are derived concerning a matrix variate extension of the Ornstein-Uhlenbeck processes. The main work on the theory of Wishart processes will be done in chapter 4, where we give a general theorem about the existence and uniqueness of Wishart processes in section 4.2. In section 4.3, we show that some soutions of (\ref{intro_wp}) can be expressed in terms of the matrix variate Ornstein-Uhlenbeck process from section 3.5 and in section 4.4 we give an algorithm to simulate  Wishart processes. Finally, we give an outlook on the applications of Wishart processes in mathematical finance in chapter 5.
\vspace{0.5cm}\\
For further readings about Wishart processes, one could have a look at the paper of \citet{bru} and for more financial applications, one could consider \citet{gou} or \citet{fonseca}, for example.

\chapter{Preliminaries}

In this section we summarize the technical prerequisites which are necessary for matrix variate stochastics.

\section{Matrix Algebra}

\begin{Definition}~\begin{enumerate}
\item Denote by $\calm_{m,n}(\bbr)$ the set of all $m\times n$ matrices with entries in $\bbr$.\\
 If $m=n$, we write $\calm_{n}(\bbr)$ instead.
\item Write $GL(p)$ for the group of all invertible elements of $\calm_{p}(\bbr)$.
\item Let $\cals_p$ denote the linear subspace of all symmetric matrices of $\calm_p(\bbr)$.
\item Let $\cals_p^{+}$ ($\cals_p^{-}$) denote the set of all symmetric positive (negative) definite matrices of  $\calm_{p}(\bbr)$.
\item Denote by $\ov{\cals_p^+}$ the closure of $\cals_p^+$ in $\calm_{p}(\bbr)$, that is the set of all symmetric positive semidefinite matrices of $\calm_{p}(\bbr)$.
\end{enumerate}\end{Definition}

Let us review some characteristics of positive (semi)definite matrices.

\begin{Theorem}[Positive definite matrices]\ 
\begin{enumerate}
\item $A\in\cals_p^{+}$ if and only if $x^\rmt Ax > 0 \; \fa x\in\bbr^p: x\neq 0 $
\item $A\in\cals_p^{+}$ if and only if $x^\rmt Ax > 0 \; \fa x\in\bbr^p: ||x||=1$
\item $A\in\cals_p^{+}$ if and only if A is orthogonally diagonalizable with positive eigenvalues, i.e. there exists an orthogonal matrix $U\in\calm_p(\bbr)$, $U U^\rmt=I_p$, such that $A=UDU^\rmt$ with a diagonal matrix  $D=\diag(\lambda_1,\ldots,\lambda_p)$, where $\lambda_i>0$, $i=1,\ldots,p$, are the positive eigenvalues of $A$
\item If $A\in\cals_p^{+}$ then $A^{-1}\in\cals_p^{+}$
\item Let $A\in\cals_p^{+}$ and $B\in\calm_{q,p}(\bbr)$ with $q\leq p$ and rank $r$. Then $BAB^\rmt\in\ov{\cals_q^+}$ and, if B has full rank, i.e. $r=q$, then $BAB^\rmt$ is even positive definite, $BAB^\rmt\in\cals_q^{+}$.
\item $A^\rmt A\in\cals_p^+$ for all $A\in GL(p)$
\item $A\in\ov{\cals_p^+}$ if and only if $x^\rmt Ax \geq 0 \; \fa x\in\bbr^p: x\neq 0 $
\item $A\in\ov{\cals_p^+}$ if and only if A is orthogonally diagonalizable with non-negative eigenvalues
\item $M^\rmt M \in\ov{\cals_p^+}$ for all $M\in\calm_p(\bbr)$
\end{enumerate}
\label{spd}\end{Theorem}

\begin{proof} See \citet[Appendix A8]{head} or \citet{fischer}. \end{proof}

On $\ov{\cals_p^+}$ we are able to define a matrix valued square root function by

\begin{Definition}[Square root of positive semidefinite matrices]\ \\
Let $A\in\ov{\cals_p^+}$. According to Theorem \ref{spd} there exists an orthogonal matrix $U\in\calm_p(\bbr)$, $U U^\rmt=I_p$, such that $A=UDU^\rmt$ with $D=\diag(\lambda_1,\ldots,\lambda_p)$, where $\lambda_i\geq 0$, $i=1,\ldots,p$.\\
Then we define the square root of $A$ by $\sqrt{A}=U\mathrm{diag}(\sqrt{\lambda_1},\ldots,\sqrt{\lambda_p})U^\rmt$, what is a positive semidefinite matrix, too.
\end{Definition}

The square root $\sqrt{A}$ is well-defined and independent of $U$, as can be seen in \citet{fischer}, for example.

\begin{Remark} If $A\in\cals_p^{+}$ then $\sqrt{A}\in\cals_p^{+}$. \end{Remark}

Now, we talk about the differentiation of matrix valued functions.

\begin{Definition}
Let $S\in\calm_p(\bbr)$. We define the differential operator
$D=(\frac{\partial}{\partial S_{ij}})_{i,j}$ for all real valued, differentiable functions $f:\calm_p(\bbr)\rightarrow\bbr$ as the matrix of all partial derivations $\frac{\partial}{\partial S_{ij}}f(S)$ of $f(S)$.
\end{Definition}

The following calculation rules are going to be helpful in the next chapters, so we state them here. In order not to lengthen this chapter, we only give outlines of the proofs.

\begin{Lemma}[Calculation rules for determinants]\ \\
For all $A,B\in \calm_p(\bbr)$, $S\in GL(p)$ and $H_t:\bbr\rightarrow GL(p)$ differentiable it holds that
\begin{enumerate}
\item $\det(AB)=\det(A)\det(B)$ and $\det(\al A)=\al^p \det(A) \quad\fa\al\in\bbr$.
\item If $A\in GL(p)$ or $B\in GL(p)$ then $\det(I_p+AB)=\det(I_p+BA)$
\item $\frac{d}{dt}\det(H_t)=\det(H_t) \tr(H_t^{-1}\frac{d}{dt}H_t)$
\item $D(\det(S))= \det(S) (S^{-1})^\rmt$
\item $\det(A)$ is the product of the eigenvalues of $A$.
\end{enumerate}
If $S$ is furthermore symmetric then
\begin{enumerate}
\setcounter{enumi}{5}
\item $D(\det(S))= \det(S) S^{-1}$
\item $\frac{\partial^2}{\partial S_{ij}\,\partial S_{kl}}(\det(S))=\det(S)[ (S^{-1})_{kl}(S^{-1})_{ij}-(S^{-1})_{ik}(S^{-1})_{lj}]$ \\
where $(S^{-1})_{ij}$ denotes the $i,j$-th entry of $S^{-1}$
\end{enumerate}
\label{lemma_calcrulesdet}\end{Lemma}

\begin{proof}
(i) and (v) can be found in every Linear Algebra book as e.g. in \citet{fischer},\\ (ii) follows from (i).\\
(iii) can be proven using the Laplace expansion and the property $\mathrm{adj}(A)=\det(A)A^{-1}$ for the adjugate matrix $\mathrm{adj}(A)$.\\
(iv) follows by using the Leibniz formula for determinants and again the adjugate property, (vi) is just a special case of (iv).\\
Finally, (vii) follows from (vi) using $\frac{\partial}{\partial S_{kl}}D(\det(S))=\det(S)[(S^{-1})_{kl}S^{-1}+\frac{\partial}{\partial S_{kl}}S^{-1}]$ and
$\frac{\partial}{\partial S_{kl}}S^{-1}=-S^{-1}(\frac{\partial}{\partial S_{kl}}S)S^{-1}$.
\end{proof}

\begin{Lemma}[Calculation rules for trace]
For all $A,B,S\in \calm_p(\bbr)$ it holds that
\begin{enumerate}
\item $\tr(AB)=\tr(BA)$ and $\tr(\al A + B)=\al \tr(A)+\tr(B)\quad\fa \al\in\bbr$
\item $\tr(A)$ is the sum of the eigenvalues of $A$.
\end{enumerate}
\label{lemma_calcrulestr}\end{Lemma}

\begin{proof}
See \citet{fischer}.
\end{proof}

\begin{Lemma}[Calculation rules for adjugate matrices]
For all $A\in GL(p)$ it holds that
\begin{enumerate}
\item $\mathrm{adj}(A)=\det(A)A^{-1}$
\item $\tr(\mathrm{adj}(A))=\det(A)\tr(A^{-1})$
\end{enumerate}
\end{Lemma}

\begin{proof}
For (i) see \citet{fischer}, (ii) is a trivial consequence of (i).
\end{proof}

The next Definition gives us a one-to-one relationship between vectors and matrices. The idea is the following: Suppose we want to transfer a theorem that holds for multivariate stochastic processes to one that holds for matrix variate stochastic processes $S$. Then we can apply the theorem to $\mathrm{vec}(S)$ and, if necessary, apply $\mathrm{vec}^{-1}$ to the resulting `multivariate processes' to get the theorem in a matrix variate version.

\begin{Definition}
Let $A\in\calm_{m,n}(\bbr)$ with columns $a_i\in\bbr^m,\,i=1,\ldots,n$. Define the function $\mathrm{vec}:\calm_{m,n}(\bbr)\rightarrow\bbr^{mn}$ via
\[ \mathrm{vec}(A)= \left( \begin{array}{c} a_1 \\ \vdots \\ a_n \end{array} \right) \]
\label{vec}\end{Definition}

Sometimes, we will also consider $\mathrm{vec}(A)$ as an element of $\calm_{mn,1}(\bbr)$.

\begin{Remark}
\[vec(A^\rmt)= \left( \begin{array}{c} \tilde{a}_1^\rmt \\ \vdots \\ \tilde{a}_m^\rmt \end{array} \right) \]
where $\tilde{a}_j\in\bbr^n,\,j=1,\ldots,m$ denote the rows of $A$.
\end{Remark}

\begin{Lemma}(Cf. \citet{grupta}[Theorem 1.2.22])
\begin{enumerate}
\item For $A,B\in\calm_{m,n}(\bbr)$ it holds that $\tr(A^\rmt B)=\mathrm{vec}(A)^\rmt \mathrm{vec}(B)$
\item Let $A\in\calm_{p,m}(\bbr)$, $B\in\calm_{m,n}(\bbr)$ and $C\in\calm_{n,q}(\bbr)$. Then we have
\[ \mathrm{vec}(AXB)=(B^\rmt \otimes A)\mathrm{vec}(X) \]
\end{enumerate}
\label{lemma_trvec}\end{Lemma}

Before we continue, we introduce a new notation: For every linear operator $\cala$ on a finite dimensional space we denote by $\sigma(\cala)$ the \emph{spectrum} of $\cala$, that is the set of all eigenvalues of $\cala$.

\begin{Lemma}
Let $A\in\calm_p(\bbr)$ be a matrix such that $0\notin\sigma(A)+\sigma(A)$. Define the linear operator 
\[ \cala:\cals_p\rightarrow\cals_p,\,X\mapsto AX+XA^\rmt \]
Then the inverse of $\cala$ is given by 
\[ \cala^{-1}=\mathrm{vec}^{-1}\circ(I_p\otimes A + A\otimes I_p)^{-1}\circ\mathrm{vec} \]
\label{lemma_linop}\end{Lemma}

\begin{proof} Can be shown with Lemma \ref{lemma_trvec} (ii), for details see \citet[p. 66]{stelzer} \end{proof}

For symmetric matrices there also exists another operator which transfers a matrix into a vector.

\begin{Definition}(Cf. \citet[Definition 1.2.7.]{grupta})
Let $S\in\cals_p$. Define the function $\mathrm{vech}:\cals_p\rightarrow\bbr^{\frac{p(p+1)}{2}}$ via
\[ \mathrm{vech}(S)= \left( \begin{array}{c} S_{11} \\ S_{12} \\ S_{22} \\ \vdots \\ S_{1p} \\ \vdots \\ S_{pp} \end{array} \right) \]
such that $\mathrm{vech}(S)$ is a vector consisting of the elements of $S$ from above and including the diagonal, taken columnwise.
\end{Definition}

Compared to $\mathrm{vec}$, the operator $\mathrm{vech}$ only takes the $\frac{p(p+1)}{2}$ distinct elements of a symmetric $p\times p$-matrix.

\section{Some Functions of Matrices}

In this section we give a brief introduction to matrix variate functions that arise in the probability density function of the noncentral Wishart distribution.

\begin{Definition}(Borel-$\sigma$-algebra, cf. \citet[p. 48]{pe})
Let $(X,\mathcal{T})$ be a topological space. The Borel-$\sigma$-algebra on X is then given by the smallest $\sigma$-algebra that contains $\mathcal{T}$ and will be denoted by $\calb(X)$.
\end{Definition}

If we work with the spaces $\bbr$, $\bbr^n$ or $\calm_{m,n}(\bbr)$ we assume that they have the natural topology, which is the set of all possible unions of open balls, given the euclidean metric. See \citet[p.22]{bvq} for details on topological spaces. We write $\calb$ instead of $\calb(\bbr)$, $\calb^n$ instead of $\calb(\bbr^n)$ and $\calb^{m,n}$ instead of $\calb(\calm_{m,n}(\bbr))$. 

\begin{Definition}[Integration]
Let $f:\calm_{m,n}(\bbr)\rightarrow\bbr$ be a $\calb^{m,n}$-$\calb$-measurable function and $M\in\calb^{m,n}$ a measurable subset of $\calm_{m,n}(\bbr)$ and let $\lambda$ denote the Lebesgue-measure on $(\bbr^{mn},\calb^{mn})$. The integral of $f$ over M is then defined by
\[ \int_M{f(X)\,dX}:=\int_M f(X)\,d(\lambda\circ\mathrm{vec})(X)=\int_{\mathrm{vec}(M)} f\circ\mathrm{vec}^{-1}(x)\,d\lambda(x) \]
We call $\lambda\circ\mathrm{vec}$ the Lebesgue-measure on $(\calm_{m,n}(\bbr),\calb^{m,n})$.
\end{Definition}

Because $\cals_p$ is isomorphic to $\bbr^{\frac{p(p+1)}{2}}$ we know that for $p\geq 2$ that $\cals_p$ is a real subspace of $\calm_p(\bbr)$ and hence a $\lambda\circ\mathrm{vec}$-Null set. Thus, we define another Lebesgue measure on the subspace of all symmetric matrices $\cals_p$. This integral is always meant if we integrate on a subset of $\cals_p$, like e.g. $\cals_p^+$, because an integral w.r.t. $\lambda\circ\mathrm{vec}$ would always be zero.

\begin{Definition}[Integration II]
Let $f:\cals_p\rightarrow\bbr$ be a $\calb(\cals_p)$-$\calb$-measurable function and $M\in\calb(\cals_p)$ a Borel measurable subset of $\cals_p$ and let $\lambda$ denote the Lebesgue-measure on $(\bbr^{\frac{p(p+1)}{2}},\calb^{\frac{p(p+1)}{2}})$. The integral of $f$ over M is then defined by 
\[ \int_M{f(X)\,dX}:=\int_M f(X)\,d(\lambda\circ\mathrm{vech})(X)=\int_{\mathrm{vech}(M)} f\circ\mathrm{vech}^{-1}(x)\,d\lambda(x) \]
We call $\lambda\circ\mathrm{vech}$ the Lebesgue-measure on $(\cals_p,\calb(\cals_p))$.
\end{Definition}

The following definition is just for convenience and can also be found in \citet{grupta}.

\begin{Definition} For $A\in\calm_p(\bbr)$ define $\etr(A) := e^{\tr(A)}$. \end{Definition}

\begin{Definition}[Matrix Variate Gamma Function]
\[ \Gamma_p(a):=\int_{\cals_p^{+}}{\etr(-A)\det(A)^{a-\frac{1}{2}(p+1)}\,dA} \quad\fa\, a>\frac{p-1}{2} \]
\end{Definition}

\citet[Theorem 1.4.1.]{grupta} show that, for $a>\frac{p-1}{2}$, the matrix variate gamma function can be expressed as a finite product of ordinary gamma functions. Thus, we do not need to worry about the existence of the matrix variate gamma function.

The definition of the Hypergeometric Function is a little bit cumbersome and needs further explanations.
Like in \citet{head} by a symmetric \emph{homogeneous polynomial} of degree $k$ in $y_1,\ldots,y_m$ we mean a polynomial which is unchanged by a permutation of the subscripts and such that every term in the polynomial has degree $k$. Denote by $V_k$ the space of all symmetric homogeneous polynomials of degree $k$ in the $\frac{p(p+1)}{2}$ distinct elements of $S\in\cals_p^{+}$. Then, $\tr(S)^k=(S_{11}+\ldots+S_{pp})^k$ is an element of $V_k$. According to \citet{grupta} $V_k$ can be decomposed into a direct sum of irreducible invariant subspaces $V_{\kappa}$ where $\kappa$ is a partition of $k$.
With a \emph{partition} $\kappa$ of $k$ we mean a $p$-tuple $\kappa=(k_1,\ldots,k_p)$ such that $k_1\geq\ldots\geq k_p\geq 0$ and $k_1+\ldots+k_p=k$.

\begin{Definition}[Zonal Polynomials]
The zonal polynomial $C_{\kappa}(S)$ is the component of $\tr(S)^k$ in the subspace $V_{\kappa}$.
\end{Definition}

The Definition implies that $\tr(S)^k=\sum_{\kappa}{C_{\kappa}(S)}$ (according to \citet{grupta}). Finally, we are able to state

\begin{Definition}[Hypergeometric Function of matrix argument]\ \\
The Hypergeometric Function of matrix argument is defined by
\beam {}_mF_n(a_1,\ldots,a_m;b_1,\ldots,b_n;S)=\sum_{k=0}^{\infty}{\sum_{\kappa}{\frac{(a_1)_{\kappa}\cdots(a_m)_{\kappa}C_{\kappa}(S)}{(b_1)_{\kappa}\cdots(b_n)_{\kappa}k!}}} \label{eq_hgf}\eeam
where $a_i,\,b_j\in\bbr$, $S$ is a symmetric $p\times p$-matrix and $\sum_{\kappa}$ the summation over all partitions $\kappa$ of $k$ and $(a)_{\kappa}=\prod_{j=1}^p \left( a-\frac{1}{2}(j-1) \right)_{k_j}$ denotes the generalized hypergeometric coefficient, with $(x)_{k_j}=x(x+1)\cdots(x+k_j-1)$. See \citet[p.30]{grupta}.
\end{Definition}

\citet[p.34]{grupta} discuss conditions for the convergence and thus well-definedness of (\ref{eq_hgf}). A sufficient condition is $m<n+1$.

\begin{Remark}
${}_nF_n(a_1,\ldots,a_n;a_1,\ldots,a_n;S)=\sum_{k=0}^{\infty}{\frac{(\tr(S))^k}{k!}}=\etr(S)$
\label{Remark_simplifyF}
\end{Remark}

The following Lemma eases later on the calculation of expectations of functions of noncentral Wishart distributed random matrices.

\begin{Lemma}
Let $Z,T\in\cals_p^{+}$. Then
\beao
\int_{\cals_p^{+}}{\etr(-ZS)\det(S)^{a-\frac{p+1}{2}}{}_mF_n(a_1,\ldots,a_m;b_1,\ldots,b_n;ST)\,dS}\\
=\Gamma_p(a)\det(Z)^{-a}{}_{m+1}F_n(a_1,\ldots,a_m,a;b_1,\ldots,b_n;Z^{-1}T)
\eeao
$\fa a>\frac{p-1}{2}$.
\label{Lemma_calcFintegral}\end{Lemma}

\begin{proof}This is a special case of \citet[Theorem 1.6.2]{grupta}\end{proof}

\chapter{Matrix Variate Stochastics}

\begin{Definition}
A quadruple $(\Omega,\calg,(\calg_t)_{t\in\bbr_+},Q)$ is called a \emph{filtered probability space} if $\Omega$ is a set, $\calg$ is a $\sigma$-field on $\Omega$, $(\calg_t)_{t\in\bbr_+}$ is an increasing family of sub-$\sigma$-fields of $\calg$ (a \emph{filtration}) and $Q$ is a probability measure on $(\Omega,\calg)$. The filtered probability space is said to satisfy the \emph{usual conditions} if
\begin{enumerate}
\item the filtration is right continuous, i.e. $\bigcap_{s>t}\calg_s=\calg_t$ for every $t\geq 0$, and
\item if $(\calg_t)_{t\in\bbr_+}$ is complete, i.e. $\calg_0$ contains all sets from $\calg$ having $Q$-probability zero.
\end{enumerate}
\end{Definition}

Throughout this thesis we assume $(\Omega,\calf,(\calf_t)_{t\in\bbr_+},P)$ to be a filtered probability space satisfying the usual conditions.

\section{Distributions}

Now we summarize a few results and definitions from \citet{grupta}.

\begin{Definition}[Random Matrix]
A $m\times n$-\emph{random matrix} $X$ is a measurable function $X:(\Omega,\calf)\rightarrow(\calm_{m,n}(\bbr),\mathcal{B}^{m,n})$.
\end{Definition}

\begin{Definition}[Probability Density Function]
A nonnegative measureable function $f_X$ such that
\[P(X\in M)=\int_{M}{f_X(A)\,dA}\quad\fa M\in\calb^{m\times n} \]
defines the\emph{ probability density function(p.d.f.)} of a $m\times n$-random matrix X.
\end{Definition}

Recall that the \emph{Radon-Nikodym} theorem (see \citet[Theorem 28.3]{pe}) says that the existence of a p.d.f. of $X$ is equivalent to saying that the distribution $P^X$, of $X$ under $P$, is absolutely continuous w.r.t the Lebesgue-measure on $(\calm_{m,n}(\bbr),\mathcal{B}^{m,n})$.

\begin{Definition}[Expectation]
Let $X$ be a $m\times n$-random matrix.
For every function $h=(h_{ij})_{i,j}:\calm_{m,n}(\bbr)\rightarrow\calm_{r,s}(\bbr)$ with  $h_{ij}:\calm_{m,n}(\bbr)\rightarrow\bbr,\,1\leq i\leq r,1\leq j\leq s$, the expected value $E(h(X))$ of $h(X)$ is an element of $\calm_{r,s}(\bbr)$ with elements 
\[ E(h(X))_{ij}=E(h_{ij}(X))=\int_{\calm_{m,n}(\bbr)}h_{ij}(A)\,P^X(dA) \]
\end{Definition}

As a matter of fact, $E(h(X))_{ij}=\int_{\calm_{m,n}(\bbr)}h_{ij}(A)f_X(A)\,dA$ if $X$ has p.d.f. $f_X$.

\begin{Definition}[Characteristic Function]\ \\
The \emph{characteristic function} of a $m\times n$-random matrix $X$ with p.d.f. $f_X$ is defined by
\beqq E[\etr(iXZ^\rmt)]=\int_{\calm_{m,n}(\bbr)}{\etr(iAZ^\rmt)f_X(A)\,dA} \label{cf}\eeqq
for every $Z\in\calm_p(\bbr)$.
\end{Definition}

\begin{Remark}\
\begin{itemize}
\item Because of $|exp(ix)|=1 \quad\fa x\in\bbr$ the above integral always exists. 
\item If the distribution of $X$ under $P$ is denoted by $P^X$, then (\ref{cf}) is the \emph{Fourier transform} of the measure $P^X$ at point $Z$ and will be denoted by $\wh{P^X}(Z)$.
\end{itemize}
\end{Remark}

From now on, with the term "$X$ in $\cals_p^+$ is a random matrix" we mean that $X$ is a random matrix with $X(\omega)\in\cals_p^+$ for almost all $\omega\in\Omega$.

\begin{Definition}[Laplace transform]\ \\
The \emph{Laplace transform} of a $p\times p$-random matrix $X$ in $\cals_p^+$ with p.d.f. $f_X$ is defined by
\beqq E[\etr(-UX)]=\int_{\cals_p^{+}}{\etr(-UA)f_X(A)\,dA}  \label{laplace}\eeqq
for every $U\in\cals_p^{+}$.
\end{Definition}

\begin{Remark}\
\begin{itemize}
\item From Lemma \ref{lemma_trvec} we see that $\tr(UA)=\mathrm{vec}(U)^\rmt \mathrm{vec}(A)$ represents an inner product on $\calm_p(\bbr)$.
\item For $A,U\in\cals_p^+$ we have that $\tr(-UA)=-\tr(\sqrt{U}A\sqrt{U})<0$, because $\sqrt{U}A\sqrt{U}$ is positive definite. Thus, the integral in (\ref{laplace}) is well-defined.
\end{itemize}
\end{Remark}

Basically, Levy's Continuity Theorem says that weak convergence of probability measures is equivalent to the  pointwise convergence of their respective Fourier transforms.

\begin{Theorem}[Levy's Continuity Theorem]
Let $(\mu_n)_{n\geq 1}$ be a sequence of probability measures on $\calm_p(\bbr)$, and let $(\wh{\mu_n})_{n\geq 1}$ denote their Fourier transforms.
\begin{enumerate}
\item If $\mu_n$ converges weakly to a probability measure $\mu$, $\mu_n\limw\mu$, then $\wh{\mu_n}(Z)\limn\wh{\mu}(Z)$ pointwise for all $Z\in\calm_p(\bbr)$.
\item If $f:\calm_p(\bbr)\rightarrow\bbr$ is continuous at $0$ and $\wh{\mu_n}(Z)\limn f(Z)$ for all $Z\in\calm_p(\bbr)$, then there exists a probability measure $\mu$ on $\calm_p(\bbr)$ such that $f(Z)=\wh{\mu}(Z)$, and $\mu_n\limw\mu$.
\end{enumerate} 
\end{Theorem}

\begin{proof}
See \citet[Theorem 19.1]{pe} for a proof of a multivariate version that can be extended to the matrix variate case easily.
\end{proof}

With the term analytic function we mean a function that is locally given by a convergent power series.

\begin{Lemma}(Cf. \citet[Lemma 3.8.4.]{mason})\ \\
Let $f$ be a real-valued analytic funtion on $\calm_p(\bbr)$ which is not identically equal to zero. Then the set $\{x\in\calm_p(\bbr):\, f(x)=0\}$ has $p^2$-dimensional Lebesgue measure zero.
\label{lemma_nullmenge}\end{Lemma}

\begin{Definition}[Covariance Matrix]
Let $X$ be a $m\times n$-random matrix and Y be a $p\times q$-random matrix. The $mn\times pq$ covariance matrix is defined as
\beao \cov(X,Y) & = & \cov(\mathrm{vec}(X^\rmt),\mathrm{vec}(Y^\rmt)) \\
& = & E[\mathrm{vec}(X^\rmt)\mathrm{vec}(Y^\rmt)^\rmt]-E[\mathrm{vec}(X^\rmt)]E[\mathrm{vec}(Y^\rmt)]^\rmt\eeao
i.e. $\cov(X,Y)$ is a $m\times p$ block matrix with blocks $\cov(\tilde{x_i}^\rmt,\tilde{y_j}^\rmt)\in\calm_{n,q}(\bbr)$ where $\tilde{x_i}$ (or resp. $\tilde{y_i}$) denote the rows of $X$ (or resp. $Y$).
\end{Definition}

\begin{Definition}[Matrix Variate Normal Distribution]
A $p\times n$-random matrix is said to have a \emph{matrix variate Normal distribution} with mean $M\in\calm_{p,n}(\bbr)$ and covariance $\Sigma\otimes\Psi$ where $\Sigma\in\cals_p^{+},\,\Psi\in\cals_n^{+}$, if $\mathrm{vec}(X^\rmt)\thicksim\caln_{pn}(\mathrm{vec}(M^\rmt),\Sigma\otimes\Psi)$ where $\caln_{pn}$ denotes the multivariate Normal distribution on $\bbr^{pn}$ with mean $\mathrm{vec}(M^\rmt)$ and covariance $\Sigma\otimes\Psi$. We will use the notation $X\thicksim\caln_{p,n}(M,\Sigma\otimes\Psi)$.
\end{Definition}

\begin{Lemma}(Cf. \citet[Theorem 2.3.1.]{grupta})
If $X\thicksim\caln_{p,n}(M,\Sigma\otimes\Psi)$, then $X^\rmt \thicksim\caln_{n,p}(M^\rmt,\Psi\otimes\Sigma)$
\label{lemma_normaltranspose}\end{Lemma}

\begin{Theorem}[Characteristic Function of the Matrix Variate Normal Distribution]\ \\
Let $X\thicksim\caln_{p,n}(M,\Sigma\otimes\Psi)$. Then the characteristic function of $X$ is given by
\beqq E[\etr(iXZ^\rmt)]=\etr\left(iZ^\rmt M - \frac{1}{2}Z^\rmt \Sigma Z\Psi\right) \label{cfnormal}\eeqq
\label{theorem_cfnd}\end{Theorem}

\begin{proof} Cf. \citet[Theorem 2.3.2]{grupta} \end{proof}

If $\Sigma$ and $\Psi$ are not positive definite, but still positive semidefinite we will use (\ref{cfnormal}) as a (generalized) definition for the matrix variate Normal distribution.

\begin{Theorem}
Let $X\thicksim\caln_{p,n}(M,\Sigma\otimes\Psi)$, $A\in\calm_{m,q}(\bbr),\,B\in\calm_{m,p}(\bbr)$ and \\ $C\in\calm_{n,q}(\bbr)$. Then $A+BXC\thicksim\caln_{m,q}(A+BMC,(B\Sigma B^\rmt)\otimes(C^\rmt \Psi C))$.
\label{theorem_scalingnd}\end{Theorem}

\begin{proof} Follows from Theorem \ref{theorem_cfnd}:
\beao
E[\etr(i(A+BXC)Z^\rmt)] & = & \etr(iAZ^\rmt)E[\etr(iX(CZ^\rmt B))] \\
& = & \etr(iAZ^\rmt)\etr\left(iCZ^\rmt B M - \frac{1}{2}CZ^\rmt B \Sigma B^\rmt Z C^\rmt \Psi\right) \\
& = & \etr\left(iZ^\rmt(A+BMC)-\frac{1}{2}Z^\rmt (B \Sigma B^\rmt) Z ( C^\rmt \Psi C)\right)
\eeao
\end{proof}

\begin{Definition}[Noncentral Wishart Distribution]
A $p\times p$-random matrix $X$ in $\cals_p^{+}$ is said to have a \emph{noncentral Wishart distribution} with parameters $p\in\bbn$, $n\geq p$, $\Sigma\in\cals_p^{+}$ and $\Theta\in\calm_p(\bbr)$ if its p.d.f is given by
\beqq
f_X(S)=\left(2^{\frac{1}{2}np}\,\Gamma_p(\frac{n}{2})\det(\Sigma)^{\frac{n}{2}}\right)^{-1}\etr\left(-\frac{1}{2}(\Theta+\Sigma^{-1}S)\right)det(S)^{\frac{1}{2}(n-p-1)} {}_0F_1\left(\frac{n}{2};\frac{1}{4}\Theta\Sigma^{-1}S\right)
\eeqq
where $S\in\cals_p^{+}$ and ${}_0F_1$ is the hypergeometric function. We write $X\thicksim\calw_p(n,\Sigma,\Theta)$.
\end{Definition}

\begin{Remark}\
\begin{itemize}
\item The requirement $n\geq p$ assures that the matrix variate gamma function is well-defined. For the case $p-1\leq n\leq p$ or resp. $n\in\{1,\ldots,p-1\}$ one could use the Laplace transform (\ref{e_ltwd}) or resp. Lemma \ref{lemma_wd} to define the Wishart distribution for this case. However, \citet[Appendix]{olkin} have shown for non-integer $n<p-1$ that (\ref{e_ltwd}) is not the Laplace transform of a probability distribution anymore.
\item If $\Sigma\in\ov{\cals_p^{+}}\backslash\cals_p^{+}$ the p.d.f. does not exist, but we can define the noncentral Wishart distribution using the characteristic function from Theorem \ref{cfwishart}
\item If $\Theta=0$, $X$ is said to have (central) Wishart distribution with parameters $p,n$ and $\Sigma\in\cals_p^{+}$ and its p.d.f. is given by
\[ \left(2^{\frac{1}{2}np}\,\Gamma_p(\frac{n}{2})\det(\Sigma)^{\frac{n}{2}}\right)^{-1}\etr\left(-\frac{1}{2}\Sigma^{-1}S\right)det(S)^{\frac{1}{2}(n-p-1)} \]
where $S\in\cals_p^{+}$ and $n\geq p$ (see \citet[Definition 3.2.1]{grupta}).
\end{itemize}
\label{centralwishart}\end{Remark}

\begin{Theorem}[Laplace Transform of the Noncentral Wishart Distribution]\ \\
Let $S \thicksim \calw_p(n,\Sigma,\Theta)$. Then the Laplace transform of S is given by
\beqq E(\etr(-US))=\det(I_p+2\Sigma U)^{-\frac{n}{2}} \etr[-\Theta(I_p+2\Sigma U)^{-1}\Sigma U]
\label{e_ltwd}\eeqq
with $U\in\cals_p^{+}$
\end{Theorem}

\begin{proof}
\beao
E(\etr{-US}) & = & \int_{\cals_p^{+}}{\etr(-US)f_S(S)\,dS} = \left(2^{\frac{1}{2}np}\,\Gamma_p(\frac{n}{2})\det(\Sigma)^{\frac{n}{2}}\right)^{-1}\etr\left(-\frac{1}{2}\Theta\right) \\
&   & \times\int_{\cals_p^{+}}\etr\left(-US-\frac{1}{2}\Sigma^{-1}S\right)det(S)^{\frac{1}{2}(n-p-1)}
{}_0F_1\left(\frac{n}{2};\frac{1}{4}\Theta\Sigma^{-1}S\right)\,dS
\eeao
With Lemma \ref{Lemma_calcFintegral} and Remark \ref{Remark_simplifyF} we get
\beao
&\int_{\cals_p^{+}}& \det(S)^{\frac{1}{2}(n-p-1)}\etr(-US-\frac{1}{2}\Sigma^{-1}S)
{}_0F_1(\frac{n}{2};\frac{1}{4}\Theta\Sigma^{-1}S)\,dS \\
& = & \Gamma_p(\frac{n}{2})\det(U+\frac{1}{2}\Sigma^{-1})^{-\frac{n}{2}}{}_1F_1(\frac{n}{2},\frac{n}{2};\frac{1}{4}(U+\frac{1}{2}\Sigma^{-1})^{-1}\Theta\Sigma^{-1}) \\
& = & \Gamma_p(\frac{n}{2})\det(\frac{1}{2}\Sigma^{-1}(I_p+2\Sigma U))^{-\frac{n}{2}}{}_1F_1(\frac{n}{2},\frac{n}{2};\frac{1}{2}(I_p+2\Sigma U)^{-1}\Theta) \\
& = & 2^{\frac{1}{2}np}\,\Gamma_p(\frac{n}{2})\det(\Sigma)^\frac{n}{2}\det(I_p+2\Sigma U)^{-\frac{n}{2}}\etr(\frac{1}{2}(I_p+2\Sigma U)^{-1}\Theta) \\
\eeao
Finally,
\beao
&&\etr(-\frac{1}{2}\Theta+\frac{1}{2}(I_p+2\Sigma U)^{-1}\Theta) \\
& = & \etr(-\frac{1}{2}\Theta(I_p-(I_p+2\Sigma U)^{-1})) \\
& = & \etr(-\frac{1}{2}\Theta(I_p+2\Sigma U)^{-1} (I_p+2\Sigma U-I_p) \\
& = & \etr(-\Theta(I_p+2\Sigma U)^{-1}\Sigma U)
\eeao
and altogether
\beao
E(\etr{-US}) & = & \int_{\cals_p^{+}}{\etr(-US)f_S(S)\,dS} = (2^{\frac{1}{2}np}\,\Gamma_p(\frac{n}{2})\det(\Sigma)^{\frac{n}{2}})^{-1}\etr(-\frac{1}{2}\Theta) \\
&   & \times\int_{\cals_p^{+}}\etr(-US-\frac{1}{2}\Sigma^{-1}S)\det(S)^{\frac{1}{2}(n-p-1)}
{}_0F_1(\frac{n}{2};\frac{1}{4}\Theta\Sigma^{-1}S)\,dS \\
& = & \det(I_p+2\Sigma U)^{-\frac{n}{2}} \etr[-\Theta(I_p+2\Sigma U)^{-1}\Sigma U]
\eeao
\end{proof}

\begin{Theorem}(Characteristic Function of the Noncentral Wishart Distribution, cf. \citet[Theorem 3.5.3.]{grupta})\ \\
Let $S \thicksim \calw_p(n,\Sigma,\Theta)$. Then the characteristic function of S is given by
\beqq E(\etr{(iZS)})=\det(I_p-2i\Sigma Z)^{-\frac{n}{2}} etr[i\Theta(I_p-2i\Sigma Z)^{-1}\Sigma Z]
\eeqq
with $Z\in\calm_p(\bbr)$
\label{cfwishart}\end{Theorem}

The next Lemma shows that the noncentral Wishart distribution is the square of a matrix variate normal distributed random matrix. Hence, it is the matrix variate extension of the noncentral chi-square distribution.

\begin{Lemma}(Cf. \citet[Theorem 3.5.1.]{grupta})\ \\
Let $X \thicksim \caln_{p,n}(M,\Sigma \otimes I_n),\,n\in\{p,p+1,\ldots\}$. Then $XX^\rmt \thicksim \calw_p(n,\Sigma,\Sigma^{-1}MM^\rmt)$.
\label{lemma_wd}\end{Lemma}

Consider $S\thicksim\calw_p(n,\Sigma,\Theta)$. With the foregoing Lemma we can interpret the parameter $\Sigma$ as a scale and the parameter $\Theta$ as a location parameter for $S$. Especially, a central Wishart distributed matrix may be thought of a matrix square of normally distributed matrices with zero mean.

\section{Processes and Basic Stochastic Analysis}

First, we make a general definition.

\begin{Definition}[Matrix Variate Stochastic Process]\ \\
A measurable function $X:\bbr_+\times\Omega\rightarrow\calm_{m,n}(\bbr),(t,\omega)\mapsto X(t,\omega)=X_t(\omega)$ is called a (matrix variate) \emph{stochastic process} if $X(t,\omega)$ is a random matrix for all $t\in\bbr_+$.\\
Moreover, $X$ is called a stochastic process in $\ov{\cals_p^+}$ if $X:\bbr_{+}\times\Omega\rightarrow\ov{\cals_p^+}$.
\end{Definition}

With $\bbr_+$ we always mean the interval of all non-negative real numbers, i.e $[0,\infty)$.\\
We can transfer most concepts from real stochastics to matrix variate stochastics, if we say a matrix variate stochastic process $X$ has property $\mathrm{P}$ if each component of $X$ has property $\mathrm{P}$. We explicitly write this down for only a few cases:

\begin{Definition}[Local Martingale]
A matrix variate stochastic process $X$ is called a \emph{local martingale}, if each component of $X$ is a local martingale, i.e. if there exists a sequence of strictly monotonic increasing stopping times $(T_n)_{n\in\bbn}$,$T_n\limas\infty$, such that $X_{\min\{n,T_n\},ij}$ is a martingale for all $i,j$.
\end{Definition}

\begin{Definition}[Matrix Variate Brownian Motion]\ \\
A matrix variate \emph{Brownian motion} $B$ in $\calm_{n,p}(\bbr)$ is a matrix consisting of independent, one-dimensional Brownian motions, i.e. $B=(B_{ij})_{i,j}$ where $B_{ij}$ are independent one-dimensional Brownian motions, $1\leq i\leq n,\,1\leq j\leq p$.
We write $B\thicksim\calbm_{n,p}$ (and $B\thicksim\calbm_{n}$ if $p=n$).
\end{Definition}

\begin{Remark}
$B_t\thicksim\caln_{n,p}(0,t I_{np})$
\end{Remark}

\begin{Theorem}
Let $W\thicksim\calbm_{n,p}$, $A\in\calm_{m,q}(\bbr),\,B\in\calm_{m,n}(\bbr)$ and $C\in\calm_{p,q}(\bbr)$. Then
$A + B W_t C \thicksim\caln_{m,q}(A,t(BB^\rmt)\otimes(C^\rmt C))$.
\end{Theorem}

\begin{proof} Follows from Theorem \ref{theorem_scalingnd} and the fact that $I_{np}=I_n\otimes I_p$ \end{proof}

\begin{Definition}[Semimartingale]
A matrix variate stochastic process $X$ is called a \emph{semimartingale} if $X$ can be decomposed into $X=X_0+M+A$ where $M$ is a local martingale and $A$ an adapted process of finite variation.
\end{Definition}

We will only consider continuous semimartingales in this thesis.\\

For a $n\times p$-dimensional Brownian motion $W\thicksim\calbm_{n,p}$, stochastic processes $X$ resp. $Y$ in  $\calm_{m,n}(\bbr)$ or resp. $\calm_{p,q}(\bbr)$ and a stopping time $T$ the matrix variate \emph{stochastic integral} on $[0,T]$ is meant to be a matrix with entries
\[ \left(\int_0^T X_t\,dW_t\,Y_t\right)_{i,j}=\sum_{k=1}^n\sum_{l=1}^p \int_0^T X_{t,ik}\,Y_{t,lj}dW_{t,kl} \quad\fa 1\leq i \leq m,\,1\leq j\leq q \]

\begin{Theorem}[Matrix Variate Itô Formula on Open Subsets]
Let $U\subseteq\calm_{m,n}(\bbr)$ be open, $X$ be a continuous semimartingale with values in $U$ and let $f:U\rightarrow \bbr$ be a twice continuously differentiable function. Then $f(X)$ is a continuous semimartingale and
\beqq f(X_t)=f(X_0) + \tr\left(\int_0^t Df(X_s)^\rmt \,dX_s\right) + \frac{1}{2}\int_0^t \sum_{j,l=1}^n \sum_{i,k=1}^m \frac{\partial^2}{\partial X_{ij}\,\partial X_{kl}} f(X_s) \, d[X_{ij},X_{kl}]_s \label{eq_ito}\eeqq
with $D=(\frac{\partial}{\partial X_{ij}})_{i,j}$
\label{theorem_ito}\end{Theorem}

\begin{proof} Follows from \citet[Theorem 3.3 and Remark 2, Chapter IV]{revuz} and Lemma \ref{lemma_trvec} \end{proof}

\begin{Corollary}
Let $X$ be a continuous semimartingale on a stochastic interval $[0,T)$ with $T=\inf\{t:X_t\notin U\}$ for an open set $U\subseteq\calm_{m,n}(\bbr)$, $X_0\in U$, and let $f:U\rightarrow \bbr$ be a twice continuously differentiable function.\\
Then $T>0$, $(f(X))_{t\in[0,T)}$ is a continuous semimartingale and (\ref{eq_ito}) holds for $t\in[0,T)$.
\end{Corollary}

\begin{proof}
For any subset $A\subseteq\calm_p(\bbr)$ , define the distance between $A$ and any point $x\in\calm_p(\bbr)$ by $d(x,A):=\inf_{z\in A} d(x,z)$. Clearly, it exists an $N=N(\omega)\in\bbn$ such that $d(X_0(\omega),\partial U)\geq \frac{1}{N}$. The continuity of $X$ and $X_0\in U$ a.s imply $N<\infty$ a.s., and as $X$ is adapted, $(T_n)_{n\geq N}$ with
\[ T_n:=\inf\{t\geq 0:\, d(X_t,\partial U)\leq \frac{1}{n}\} \]
defines a sequence of positive stopping times. For every $n\geq N$, the stopped process $X^{T_n}:=X_{t\land T_n}$ is a continuous semimartingale with values in $U$, thus Theorem \ref{theorem_ito} can be applied to see that $f(X^{T_n})$ is a continuous semimartingale and (\ref{eq_ito}) holds for $X^{T_n}$. Because $T_n$ converges to $T=\inf\{t\geq 0:\, X_t\in\partial U\}$, we can conclude $T>0$ and that $(f(X))_{t\in[0,T)}$ is a continuous semimartingale and (\ref{eq_ito}) holds for $(X)_{t\in[0,T)}$.
\end{proof}

With the Definition of a `Matrix Quadratic Covariation' we are able to state the matrix variate partial integration formula in a handy way.

\begin{Definition}[Matrix Quadratic Covariation]
For two semimartingales $A\in\calm_{d,m}(\bbr),\,B\in\calm_{m,n}(\bbr)$ the matrix variate \emph{quadratic covariation} is defined by
\[ [A,B]_{t,ij}^M=\sum_{k=1}^m [A_{ik},B_{kj}]_t\in\calm_{d,n}(\bbr) \]
\end{Definition}

\begin{Theorem}[Matrix Variate Partial Integration](Cf. \citet[Lemma 5.11.]{barndorff})
Let $A\in\calm_{d,m}(\bbr),\,B\in\calm_{m,n}(\bbr)$ be two semimartingales. Then the matrix product $A_t B_t\in\calm_{d,n}(\bbr)$ is a semimartingale and
\[ A_t B_t = \int_0^t A_{t-}\,dB_t + \int_0^t dA_t\,B_{t-} + [A,B]_{t}^M \]
where $A_{t-}$ denotes the limit from the left of $A_t$.
\label{partialintegration}\end{Theorem}

\section{Stochastic Differential Equations}

When we talk about (matrix variate) stochastic differential equations, we can classify between two different kind of solutions: Weak and strong solutions. Intuitively, a strong solution is constructed from a given Brownian motion and hence a `function' of that Brownian motion.

\begin{Definition}(Weak and Strong Solutions, cf. \citet[Chapter IX, Definition 1.5]{revuz})
Let $(\Omega,\calg,(\calg_t)_{t\in\bbr_+},Q)$ be a filtered probability space satisfying the usual conditions and consider the stochastic differential equation
\beqq dX_t=b(t,X_t)\,dt + \sigma(t,X_t)\,dW_t,\quad X_0=x_0 \label{eq_sde}\eeqq
where $b:\bbr_+\times\calm_{m,n}(\bbr)\rightarrow\calm_{m,n}(\bbr)$ and $\sigma:\bbr_+\times\calm_{m,n}(\bbr)\rightarrow\calm_{m,p}(\bbr)$ are measurable functions, $x_0\in\calm_{m,n}(\bbr)$ and $W$ is a $p\times n$-dimensional Brownian motion.
\begin{itemize}
\item A pair $(X,W)$ of $\calf_t$-adapted continuous processes defined on $(\Omega,\calg,(\calg_t)_{t\in\bbr_+},Q)$ is called\\ a \emph{solution} of (\ref{eq_sde}) on $[0,T)$, $T>0$, if $W$ is a $(\calg_t)_{t\in\bbr_+}$-Brownian motion and 
\[ X_t=x_0+\int_0^t b(s,X_s)\,ds+\int_0^t \sigma(s,X_s)\,dW_s \quad\fa t\in[0,T) \]
\item Moreover, the pair $(X,W)$ is said to be a \emph{strong solution} of (\ref{eq_sde}), if $X$ is adapted to the filtration $(\calg_t^W)_{t\in\bbr_+}$, where $\calg_t^W=\sigma_c(W_s,s\leq t)$ is the $\sigma$-algebra generated from $W_s$,$s\leq t$, that is completed with all $Q$-null sets from $\calg$.
\item A solution $(X,W)$ which is not strong will be termed a \emph{weak solution} of (\ref{eq_sde}).
\end{itemize}
\end{Definition}

To define the \emph{probability law $P^X$ of a stochastic process}  $X:\bbr_+\times\Omega\rightarrow\calm_{m,n}(\bbr)$, we consider, according to \citet{oek}, $X$ as a random matrix on the functional space $(\calm_{m,n}(\bbr)^{\bbr_+},\wh{\calf})$ with $\sigma$-algebra $\wh{\calf}$ that is generated by the cylinder sets
\[ \{\omega\in\Omega:\,X_{t_1}(\omega)\in F_1,\ldots,X_{t_k}(\omega)\in F_k\} \]
where $F_i\subset\calm_{m,n}(\bbr)$ are Borel sets and $k\in\bbn$.\\ $X:(\Omega,\calf,P)\rightarrow(\calm_{m,n}(\bbr)^{\bbr_+},\wh{\calf})$ is then a measurable function with law 
$P^X$.\\ 

\begin{Definition}(Uniqueness, \citet[Chapter IX, Definition 1.3]{revuz})
Consider again the stochastic differential equation (\ref{eq_sde}).
\begin{itemize}
\item It is said that \emph{pathwise uniqueness} holds for (\ref{eq_sde}), if for every two solutions $(X,W)$ and $(X',W')$ defined on the same filtered probability space, $X_0=X_0'$ and $W=W'$ a.s. implies that $X$ and $X'$ are \emph{indistinguishable}, i.e. for $P$-almost all $\omega$ it holds that $X_t(\omega)=X_t'(\omega)$ for every $t$, or equivalently
\[ P\left(\sup_{t\in[0,\infty)} |X_t-X_t'|>0\right)=0 \]
\item There is \emph{uniqueness in law} for (\ref{eq_sde}), if whenever $(X,W)$ and $(X',W')$ are two solutions with possibly different Brownian motions $W$ and $W'$ (in particular if $(X,W)$ and $(X',W')$ are defined on two different filtered probability spaces) and $X_0\eqd X_0'$, then the laws $P^X$ and $P^{X'}$ are equal. In other words, $X$ and $X'$ are two versions of the same process, i.e. they have the same finite dimensional distributions (see \citet[Chapter I, Definition 1.6]{revuz}).
\end{itemize}
\end{Definition}

\begin{Remark}\
\begin{enumerate}
\item As \citet[Proposition 1]{yamada1} have shown, pathwise uniqueness implies uniqueness in law, which is not true conversely.
\item If pathwise uniqueness holds for (\ref{eq_sde}), then every solution of (\ref{eq_sde}) is strong (see \citet[Theorem 1.7]{revuz}).
\item The definition of pathwise uniqueness implies that there exists at most one strong solution for (\ref{eq_sde}) up to indistinguishability.
\item According to \citet[p. 59f.]{sko}, the stochastic differential equation (\ref{eq_sde}) always has a weak (but not necessarily unique) solution if $b$ and $\sigma$ are continuous functions. If in this situation pathwise uniqueness holds for (\ref{eq_sde}), then there exists a unique strong solution up to indistinguishability.
\end{enumerate}\label{remark_pu}\end{Remark}

Similar to the theory of ordinary differential equations, the function on the right hand side of a stochastic differential equation being locally Lipschitz is sufficient in order to guarantee (strong) existence on a nonempty (stochastic) interval and (pathwise) uniqueness.

\begin{Theorem}(Existence of Solutions of SDEs driven by a continuous Semimartingale, cf. \citet[Theorem 6.7.3.]{stelzer})
Let $U$ be an open subset of $\calm_{d,n}(\bbr)$ and $(U_n)_{n\in\bbn}$ a sequence of convex closed sets such that $U_n\subset U,\,U_n\subseteq U_{n+1} \fa n\in\bbn$ and $\bigcup_{n\in\bbn} U_n = U$. Assume that $f:U\rightarrow\calm_{d,m}(\bbr)$ is a locally Lipschitz function and $Z$ in $\calm_{m,n}(\bbr)$ is a continuous semimartingale. Then for each $U$-valued $\calf_0$-measurable initial value $X_0$ there exist a stopping time $T$ and a \emph{unique $U$-valued strong solution} $X$ to the stochastic differential equation
\beqq dX_t=f(X_t)dZ_t \label{eq_exsde}\eeqq
up to the time $T>0$ a.s., i.e. on the stochastic interval $[0,T)$.
On $T<\infty$ we have that either $X$ hits the boundary $\partial U$ of $U$ at $T$, i.e. $X_T\in\partial U$ or explodes, i.e. $\limsup_{t\rightarrow T,t<T} ||X_t||=\infty$. If $f$ satifies the \emph{linear growth condition}
\[ ||f(X)||^2 \leq K(1+||X||^2) \]
with some constant $K\in\bbr_+$, then no explosion can occur.
\label{theorem_exsde}\end{Theorem}

\begin{Remark}
With the term unique it is meant that there holds pathwise uniqueness for (\ref{eq_exsde}). In other words, every two solutions $(X,Z)$ and $(X',Z)$ of (\ref{eq_exsde}) defined on the same probability space and with the same continuous semimartingale $Z$ and the same initial value are indistinguishable.
\end{Remark}

\begin{Definition}(Local Lipschitz, cf. \citet[Definition 6.7.1.]{stelzer})
Let $(U,||.||_U)$, $(V,||.||_V)$ be two normed spaces and $W\subseteq U$ be open. Then a function $f:W\rightarrow V$ is called locally Lipschitz, if for every $x\in W$ there exists an open neighbourhood $\mathcal{U}(x)\subset W$ and a constant $C(x)\in\bbr_+$ such that
\beqq ||f(z)-f(y)||_V\leq C(x)||z-y||_U \quad\fa z,y\in \mathcal{U}(x) \eeqq
$C(x)$ is said to be a local Lipschitz coefficient.
If there is a $K\in\bbr_+$ such that $C(x)=K$ can be chosen for all $x\in W$, $f$ is called globally Lipschitz.
\end{Definition}

If the process $Z$ from Theorem \ref{theorem_exsde} is a continous Lévy process (Brownian motion with drift), it can be shown that the solution $X$ of (\ref{eq_exsde}) is a Markov process. Before we state this in a theorem, we give a general definition of the term 'Markov process' and a necessary technical condition on our probability space.

\begin{Definition}(Markov Process, cf. \citet[Definition 6.7.7.]{stelzer})
Let $U\subseteq\calm_p(\bbr)$ be open and $Z$ be a process with values in $U$ which is adapted to a filtration $(\calf_t)_{t\in\bbr_+}$.
\begin{enumerate}
\item $Z$ is called a \emph{Markov process} with respect to $(\calf_t)_{t\in\bbr_+}$, if
\[ E(g(Z_u)|\calf_t)=E(g(Z_u)|Z_t) \]
for all $t\in\bbr_+,\,u\geq t$ and $g:U\rightarrow\bbr$ bounded and Borel measurable.
\item Let $Z$ be a Markov process and define for all $s,t\in\bbr_+,s\leq t$ the \emph{transition functions}
$P_{s,t}(Z_s,g):=E(g(Z_t)|Z_s)$ with $g:U\rightarrow\bbr$ bounded and Borel measurable. If 
\[ P_{s,t}=P_{0,t-s}=:P_{t-s} \textit{ for all } s,t\in\bbr_+,s\leq t \]
then $Z$ is said to be a \emph{time homogeneous Markov process}.
\item A time homogeneous Markov process is called a \emph{strong Markov} process, if
\[ E(g(Z_{T+s})|\calf_T)=P_s(Z_T,g)=E(g(Z_{T+s})|Z_T) \]
for all $g:U\rightarrow\bbr$ bounded and Borel measurable and a.s. finite stopping times $T$.
\end{enumerate}
\end{Definition}

At the beginning of this chapter we assumed $(\Omega,\calf,(\calf_t)_{t\geq 0},P)$ to be a filtered probability space where $(\calf_t)_{t\geq 0}$ is right continuous and all $\calf_t$ are complete. Now, we need to \emph{enlarge} our given probability space in order to get arbitrary initial values. 

\begin{Definition}(Enlargement of a probability space, cf. \citet{protter})\ \\
Let $(\Omega,\calf,(\calf_t)_{t\geq 0},P)$ be a filtered probability space where $(\calf_t)_{t\geq 0}$ is right continuous, i.e. $\bigcap_{s>t}\calf_s=\calf_t$ for every $t\geq 0$, and all $\calf_t$ are completed with sets from $\calf$ having $P$-probability zero and $U\subseteq\calm_p(\bbr)$ be an open subset of $\calm_p(\bbr)$.\\
The probability space $(\ov{\Omega},\ov{\calf},(\ov{\calf}_t)_{t\geq 0},(\ov{P}^y)_{y\in U})$ with
\[ \ov{\Omega}=U\times\Omega,\quad \ov{\calf}_t=\bigcap_{u>t}\sigma(\calb(U)\times\calf_u),\quad \ov{\calf}=\sigma(\calb(U)\times\calf),\quad \ov{P }^y=\delta_y\times P \]
is called the \emph{Enlargement} of $(\Omega,\calf,(\calf_t)_{t\geq 0},P)$. The measure $\delta_y$ is the Dirac measure w.r.t. $y\in U$, i.e. for $A\subseteq U$ it holds that $\delta_y(A)=1$ if and only if $y\in A$. A random matrix $Z:\Omega\rightarrow\calm_p(\bbr)$ is \emph{extended} to $\ov{\Omega}$ by setting $Z((y,\omega))=Z(\omega)$ for all $(y,\omega)\in\ov{\Omega}$.
\end{Definition}

Eventually we are able to state

\begin{Theorem}(Markov Property of Solutions of SDEs driven by a Brownian Motion, cf. \citet[Theorem 6.7.8.]{stelzer})
Recall the Assumptions of Theorem \ref{theorem_exsde} (without the linear growth condition), but let this time $Z$ be a Brownian motion with drift in $\calm_{m,n}(\bbr)$, i.e. $Z\thicksim\calbm_{m,n}$. Suppose further that $T=\infty$ for every initial value $x_0\in U$. Consider the enlarged probability space $(\ov{\Omega},\ov{\calf},(\ov{\calf}_t)_{t\geq 0},(\ov{P}^y)_{y\in U})$ and define $X_0$ by $X_0((y,\omega)):=y$ for all $y\in U$. Then the unique strong solution $X$ of
\beqq dX_t=f(X_t)dZ_t \eeqq
is a time homogeneous strong Markov process on $U$ under every probability measure of the family $(\ov{P}^y)_{y\in U}$.
\label{markovsolution}\end{Theorem}

The Girsanov Theorem is a powerful tool that gives us the opportunity to construct a new probability measure $\wh{Q}$ such that a drift changed $P$-Brownian motion (that is not a Brownian motion under the probability measure $P$ anymore) is a Brownian motion under the new measure $\wh{Q}$. Before we state the Girsanov Theorem we make

\begin{Definition}[Stochastic Exponential]
Let $X$ be a stochastic process. The unique strong solution $Z=\mathcal{E}(X)$ of
\beqq dZ_t=Z_t\,dX_t,\quad Z_0=1 \label{eq_se}\eeqq
is called stochastic exponential of $X$.
\end{Definition}

From Theorem \ref{theorem_exsde} we get immediately that (\ref{eq_se}) has a unique strong solution.

\begin{Theorem}[Matrix Variate Girsanov Theorem]
Let $T>0$, $B\thicksim\calbm_p$ and $U$ be an adapted, continuous stochastic process with values in $\calm_p(\bbr)$ such that
\beqq \left(\mathcal{E}\left(\tr\left(-\int_0^t U_s^\rmt\,dB_s\right)\right)\right)_{t\in[0,T]} \label{girsanov_martingale}\eeqq
is a martingale, or, what is sufficient for (\ref{girsanov_martingale}), but not necessary, such that the \emph{Novikov condition} is satisfied
\beqq E\left(\etr\left(\frac{1}{2}\int_0^T U_t^\rmt U_t\,dt\right)\right)<\infty \label{eq_novikov}\eeqq
Then 
\beqq \wh{Q}=\int\mathcal{E}\left(\tr\left(-\int_0^T U_t^\rmt\,dB_t\right)\right)\,dP \label{eq_q}\eeqq
is an equivalent probability measure, and
\beqq \wh{B_t}=\int_0^t U_s\,ds+B_t \label{eq_whb}\eeqq
is a $\wh{Q}$-Brownian motion on $[0,T)$.
\label{girsanov}\end{Theorem}

\begin{proof}
We use the multivariate Girsanov Theorem of \citet[Corollary 16.25]{kalle}. \\
The process $V:=\mathrm{vec}(-U)\in\bbr^{p^2}$ is according to \citet[Proposition 4.8, Chapter I]{revuz} progressively measurable and $B^v:=\mathrm{vec}(B)$ is a Brownian motion on $\bbr^{p^2}$. The Novikov condition is then
\beao
E\left(exp\left(\frac{1}{2}\int_0^T \sum_{i=1}^{p^2} V_{t,i}^2\,dt\right)\right) &=& E\left(exp\left(\frac{1}{2}\int_0^T \tr(U_t^\rmt U_t)\,dt\right)\right) \\
&=& E\left(\etr\left(\frac{1}{2}\int_0^T U_t^\rmt U_t\,dt\right)\right)<\infty
\eeao
With \citet[Corollary 16.25]{kalle} the new measure $Q$ is then
\beao
\wh{Q}=\int\mathcal{E}\left(\int_0^T V_t^\rmt\,dB^v_t\right)\,dP=\int\mathcal{E}\left(\tr\left(-\int_0^T U_t^\rmt\,dB_t\right)\right)\,dP
\eeao
and
\[ \wt{B^v_t}=B^v_t-\int_0^t V_s\,ds\]
is a $\wh{Q}$-Brownian motion for $t\in[0,T]$ and with values in $\bbr^{p^2}$.
Hence,
\[ \wh{B_t}=\mathrm{vec}^{-1}(\wt{B^v_t})=\int_0^t U_s\,ds+B_t \]
is a $\wh{Q}$-Brownian motion for $t\in[0,T]$ and with values in $\calm_p(\bbr)$.
\end{proof}

\section{McKean's Argument}

L\'evy's Theorem gives us an easy way to decide whether a continuous local martingale is a Brownian motion.

\begin{Theorem}[L\'evy's Theorem]\ \\
Let $B$ be a $p\times p$ dimensional continuous local martingale such that
\[ [B_{ij},B_{kl}]_t= \left\{ \begin{array}{cc} t, \, & \textnormal{if } i=k \textnormal{ and } j=l  \\ 0, & \textnormal{else} \end{array} \right\} \]
for all $i,j,k\in \{1,\ldots,p\}$ and $B_0=0$. Then B is a $p\times p$ dimensional Brownian motion, $B\thicksim \calbm_p$.
\label{theorem_levy}\end{Theorem}

\begin{proof}
It suffices to show that $X:=\mathrm{vec}(B)$ is a $p^2$-dimensional Brownian motion.\\
Let $a,b\in\{1,\ldots,p^2\}$. Then, there exist $i,j,k,l\in\{1,\ldots,p\}$ such that $a=p(j-1)+i$ and $b=p(l-1)+k$. Clearly, $i=k$ and $j=l$ imply $a=b$, and conversely, if we assume $a=b$, then $i\neq k$ or resp. $j\neq l$ imply contradictions. Thus,
\[ [X_a,X_b]=[B_{ij},B_{kl}]=t \mathbf{1}_{\{i=k\}}\mathbf{1}_{\{j=l\}} = t \mathbf{1}_{\{a=b\}} \]
With \citet[Theorem 40,Chapter 2]{protter} we can conclude that $X$ is a $p^2$-dimensional Brownian motion.
\end{proof}

For every finite stopping time $\tau_0$ we say $M$ is a local martingale on the (stochastic) interval $[0,\tau_0]$ if the stopped process $(M_{\inf\{t,\tau_0\}})_{t\in\bbr_+}$ is a local martingale. The next Theorem shows us that every stopped continous local martingale is a stopped, time-changed Brownian motion. There also exists a version of Theorem \ref{theorem_dambis} without stopping times but with the additional assumption that $\lim_{t\rightarrow\infty}[M,M]_t=\infty$, see the Dambis-Dubins-Schwarz-Theorem in \citet[Theorem 1.6, Chapter V]{revuz}. 

\begin{Theorem}
Let $0<\tau_0<\infty$ be a stopping time and $M$ be a continuous local martingale on the interval $[0,\tau_0]$, which is not identically equal to zero. If we set 
\[ T_t := \inf \{s:\,[M,M]_s > t \} \textnormal{ (with the convention } \inf\{\emptyset\}=\infty) \]
then $B_t:=M_{T_t}$ is a stopped $\calf_{T_t}$-Brownian motion on the stochastic interval $[\,0,[M,M]_{\tau_0})$, $[M,M]_{\tau_0}>0$ a.s., i.e $B$ is a Brownian motion w.r.t. the $\sigma$-algebra 
\[ \calf_{T_t}=\{A\in\calf:\,A\cap\{T_t\leq t\}\in\calf_t \quad t\in\bbr_+\} \]
that is stopped at $[M,M]_{\tau_0}$. Furthermore, it holds that $M_t=B_{[M,M]_t}$, i.e. $M$ is a stopped, time-changed Brownian motion.
\label{theorem_dambis}\end{Theorem}

\begin{proof}
Observe that $T_t$ is a generalized inverse function of $[M,M]_t$. Because $M$ is continuous, and hence $[M,M]$, too, it holds that $[M,M]_{T_t}=t$, but still we only have $T_{[M,M]_t}\geq t$ in general (and strict equality for every $t$ if and only if $M$ was strict monotonic on the entire interval). See \citet[p.7-8]{revuz} for details. \\
Suppose $[M,M]_{\tau_0}=0$. Because $[M,M]$ is an increasing process, $[M,M]_t=0$ for all $t\in[0,\tau_0]$, thus $E([M,M]_t)=0$ for all $t\in[0,\tau_0]$. With \citet[Chapter II, Corollary 3]{protter} it follows $E(M_t^2)=0$ for all $t\in[0,\tau_0]$, hence $M$ is equal to zero, which is a contradiction.\\
The set $\{T_t\}_{t\in\bbr_+}$ is a family of stopping times, because $M$ is adapted. Observe that
\[  [M,M]_{\tau_0}>t \Rightarrow T_t\leq\tau_0 \]
Hence $T_t<\infty$ as $\tau_0<\infty$. From \citet[Proposition 4.8 and 4.9, Chapter I]{revuz} we know that the stopped process $M_{T_t}$ is $\calf_{T_t}$-measurable. Furthermore,
\[ [B,B]_t=[M,M]_{T_t}=t \]
as mentioned above. With Theorem \ref{theorem_levy} we can can conclude that $B$ is a stopped \\$\calf_{T_t}$-Brownian motion on $[\,0,[M,M]_{\tau_0})$.\\
To prove that $M$ is a time-changed Brownian motion, observe that $T_{[M,M]_t}>t$ if and only if $[M,M]$ is constant on $[t,T_{[M,M]_t}]$. As $M$ and $[M,M]$ are constant on the same intervals, we have $M_{T_{[M,M]_t}}=M_t$ and thus $B_{[M,M]_t}=M_t$.
\end{proof}

With the foregoing Theorem we can form an argument that will allow us in turn to proof some result on the existence of the Wishart process in the next chapter.

\begin{Theorem}(McKean's Argument, cf. \citet[p.47,Problem 7]{mckean})\\
Let $r$ be a real-valued continous stochastic process such that $P(r_0>0)=1$ and $h:\bbr_+\rightarrow\bbr$ such that
\begin{itemize}
\item $h(r)$ is a continuous local martingale on the interval $[0,\tau_0)$ for $\tau_0:=\inf\{s:\,r_s=0\}$
\item $\lim_{x\to 0,x>0} h(x)=\infty$ or resp. $\lim_{x\to 0,x>0} h(x)=-\infty$
\end{itemize}
Then $\tau_0=\infty$ almost surely, i.e. $r_t(\omega)>0 \quad\fa t\in\bbr_+$ for almost all $\omega$.
\label{theorem_mckean}\end{Theorem}

\begin{Remark}
If $h(r_t)=\int_0^t f(s,r_s)\,dW_s$ with a Brownian motion $W$, then for $h(r)$ being a continuous local martingale on $[0,\tau_0)$ it is sufficient that $f$ is square integrable on $[0,\tau_0)$.
\end{Remark}

\begin{proof}
Suppose $\tau_0<\infty$. \\
Define $M:=h(r)$ and $T_t$ as in Theorem \ref{theorem_dambis}. $M$ is a continuous local martingale on $[0,\tau_0]$ and with Theorem \ref{theorem_dambis} we can conclude that $B_t:=M_{T_t}$ is a stopped Brownian motion on $[\,0,[M,M]_{\tau_0})$. Observe, that $r_0>0$ and the continuity of $r$ impliy $\tau_0>0$.\\
Consider the case $\lim_{x\to 0,x>0} h(x)=\infty$. On the interval $[0,\tau_0)$, the process $h(r_t)$ takes every value in $[h(r_0),\infty)$, especially $[h(r),h(r)]_{\tau_0}=[M,M]_{\tau_0}>0$. For $\tau_0<\infty$ we have
\[ t\rightarrow [M,M]_{\tau_0} \Rightarrow T_t\rightarrow\tau_0 \Rightarrow r_{T_t}\rightarrow 0 \Rightarrow B_t=h(r_{T_t})\rightarrow \infty \]
There are two cases:
\begin{itemize}
\item $[M,M]_{\tau_0}<\infty$, what is impossible because a Brownian motion can not go to infinity in finite time, see \citet[Law of the iterated logarithm, Corollary 1.12, Chapter II]{revuz}.
\item $[M,M]_{\tau_0}=\infty$, which implies that $B$ is a Brownian motion on $\bbr_+$ that converges to $\infty$  almost surely and that is a contradiction because,
\[ P(B_t\leq 0 \textnormal{ infinitely often for } t\rightarrow\infty)=1, \]
i.e $B$ has infinite oscillations.
\end{itemize}
As both cases imply contradictions, we conclude $\tau_0=\infty$.\\
The case for $\lim_{x\to 0,x>0} h(x)=-\infty$ is	analogous.
\end{proof}

\section{Ornstein-Uhlenbeck Processes}

We give the following definition according to \citet{bru} such that we are later able to show that some solutions of the Wishart SDE can be constructed out of a matrix variate Ornstein-Uhlenbeck process.

\begin{Definition}[Matrix Variate Ornstein-Uhlenbeck Process]\ \\
Let $A,B\in\calm_p(\bbr)$, $x_0\in\calm_{n,p}(\bbr)$ a.s. and $W\thicksim\calbm_{n,p}$. A solution $X$ of
\beqq
dX_t = X_tB\,dt + dW_t\,A,\quad X_0 = x_0
\label{eq_oup}\eeqq
is called $n\times p$-dimensional Ornstein-Uhlenbeck process. We write $X\thicksim\caloup_{n,p}(A,B,x_0)$.
\end{Definition}

As $X\mapsto XB$ and $X\mapsto A$ are trivially globally Lipschitz and satisfy the linear growth condition we know from Theorem \ref{theorem_exsde} that (\ref{eq_oup}) has a unique strong solution on the entire interval $[0,\infty)$. Luckily, we are even able to give an explicit formula for the solution of (\ref{eq_oup}).

\begin{Theorem}[Existence and Uniqueness of the Ornstein-Uhlenbeck Process]\ \\
For a Brownian motion $W\thicksim\calbm_{n,p}$, the unique strong solution of (\ref{eq_oup}) is given by
\beqq X_t = x_0e^{Bt}+\left(\int_0^t dW_s\,A e^{-Bs}\right)e^{Bt} \label{eq_solution_oup}\eeqq
\end{Theorem}

\begin{proof}\beao
dX_t & = & d(x_0e^{Bt}) + d\left(\left(\int_0^t dW_s\,A e^{-Bs}\right)e^{Bt}\right) \\
& = & x_0e^{Bt}B\,dt + dW_t\,A e^{-Bt}e^{Bt} + \left(\int_0^t dW_s\,A e^{-Bs}\right)e^{Bt}B\,dt + \underbrace{d\left[\int_0^{\cdot} dW_sAe^{-Bs},e^{B\cdot}\right]_t^M}_{=0} \\
& = & \left(x_0e^{Bt}+\left(\int_0^t dW_s\,A e^{-Bs}\right)e^{Bt}\right)B\,dt + dW_t\,A \\
& = & X_tB\,dt + dW_t\,A
\eeao
Furthermore, (\ref{eq_solution_oup}) is a strong solution by construction.
\end{proof}

The following Lemma will be useful for determining the conditional distribution of the Ornstein-Uhlenbeck process.

\begin{Lemma}
Let $W\thicksim\calbm_{n,p}$ and $X:\bbr_+\rightarrow\calm_{p,m}(\bbr),\,t\mapsto X_t$ be a square integrable, deterministic function. Then
\[ \int_0^t dW_s\,X_s \thicksim \caln_{n,m}(0,I_n\otimes \int_0^t X_s^\rmt X_s\,ds) \]
\label{lemma_integral}\end{Lemma}

\begin{proof}
$H_t:=\int_0^t dW_s\,X_s$. $H_t\in\calm_{n,m}(\bbr)$ is a $\mathcal{L}^2$-limit of normal distributed, independent, random variables and therefore normally distributed (Cf. \citet[Proof of Theorem 5.2.1 and Theorem A.7.]{oek}). Furthermore, $H_t$ is a martingale with a.s. initial value zero and thus expectation zero. \\
$\cov(H_t,H_t)$ is a block diagonal matrix, because different rows of $W$ are independent. Hence, $\cov(H_t,H_t)$ is a tensor product of $I_n$ with another $m$-dimensional matrix. Furthermore, all rows of $W$ are identically distributed such that we only need to consider row 1 of $H_t$, i.e. the first block matrix $D_t^1$ of $\cov(H_t,H_t)$. We have that \\
\[ H_{t,1i}= \sum_{k=1}^p \int_0^t dW_{s,1k}\,X_{s,ki} \]
and thus
\beao D_{t,ij}^1 & = & \cov(H_{t,1i},H_{t,1j}) \\
& = & \sum_{k_i,k_j=1}^p \cov(\int_0^t dW_{s,1k_i}\,X_{s,k_ii},\int_0^t dW_{s,1k_j}\,X_{s,k_jj}) \\
& = & \sum_{k_i,k_j=1}^p \int_0^t X_{s,k_ii} X_{s,k_jj}\,ds \mathbf{1}_{\{k_i=k_j\}} \\
& = & \sum_{k=1}^p \int_0^t X_{s,ki} X_{s,kj}\,ds = \int_0^t (X_s^\rmt X_s)_{ij}\,ds
\eeao
where we used the fact that $\cov(\int f_1(s)\,dW_{s,1k_i},\int f_2(s)\,dW_{s,1k_j})= \int f_1(s)f_2(s)\,ds \,\mathbf{1}_{\{k_i=k_j\}}$ for deterministic, square integrable real functions $f_1,f_2$.
\end{proof}

Now we are able to show that the (matrix variate) Ornstein-Uhlenbeck process is (matrix variate) normally distributed. Later we will see that some Wishart processes are therefore `square' normally distributed, that is by Lemma \ref{lemma_wd} a noncentral Wishart distribution.\\
To clarify terms, with \emph{distribution of a stochastic process} $X$ we always mean from now on the distribution of $X_t$ at time $t$ conditional on the initial value $X_0$. Although it may be confusing, the terms distribution of a stochastic process and law of a stochastic process have a different meaning in this thesis.

\begin{Theorem}[Distribution of the Matrix Variate Ornstein-Uhlenbeck Process]\ \\
Let $X\thicksim\caloup_{n,p}(A,B,x_0)$ with a matrix $B\in\calm_p(\bbr)$ that satisfies $0\notin-\sigma(B)-\sigma(B)$. Then the distribution of $X$ is given by
\[ X_t | x_0 \thicksim\caln_{n,p}(x_0e^{Bt},I_n \otimes [\cala^{-1}(A^\rmt A)-\cala^{-1}(e^{B^\rmt t}A^\rmt Ae^{Bt})]) \]
where $\cala^{-1}$ is the inverse of the linear operator $\cala:\cals_p\rightarrow\cals_p,X\mapsto-B^\rmt X-XB$.
\label{theorem_doup}\end{Theorem}

\begin{proof}
Define $H_t = \int_0^t dW_sAe^{-Bs}$. Lemma \ref{lemma_integral} shows us that
\[ H_t\thicksim\caln_{n,p}(0,I_n\otimes \int_0^t e^{-B^\rmt s}A^\rmt Ae^{-Bs}\,ds) \]
Define the operator $\cala:X\mapsto -B^\rmt X-XB$ as in Lemma \ref{lemma_linop}. Then we have
\[ \frac{d}{ds}(e^{-B^\rmt s}A^\rmt Ae^{-Bs})=\cala(e^{-B^\rmt s}A^\rmt Ae^{-Bs}) \]
and hence
\beqq \int_0^t e^{-B^\rmt s}A^\rmt Ae^{-Bs}\,ds=\cala^{-1}(e^{-B^\rmt s}A^\rmt Ae^{-Bs})|_{s=0}^t \label{ainvint}\eeqq
The equality $X_t = x_0e^{Bt} + H_te^{Bt}$ and Theorem \ref{theorem_scalingnd} gives us
\[ X_t | x_0 \thicksim\caln_{n,p}(x_0e^{Bt},I_n \otimes e^{B^\rmt t} \cala^{-1}(e^{-B^\rmt s}A^\rmt Ae^{-Bs})|_{s=0}^t e^{Bt}) \]
Because $\cala$ is a linear operator, $\cala^{-1}$ is also linear and we can simplify
\[ \cala^{-1}(e^{-B^\rmt s}A^\rmt Ae^{-Bs})|_{s=0}^t=\cala^{-1}(e^{-B^\rmt t}A^\rmt Ae^{-Bt})-\cala^{-1}(A^\rmt A) \]
and, because $\cala^{-1}$ is the integral from (\ref{ainvint}),
\[
e^{B^\rmt t} (\cala^{-1}(e^{-B^\rmt t}A^\rmt Ae^{-Bt})-\cala^{-1}(A^\rmt A)) e^{Bt}=
\cala^{-1}(A^\rmt A)-\cala^{-1}(e^{B^\rmt t}A^\rmt Ae^{Bt}) \]
\end{proof}

\begin{Theorem}[Stationary Distribution of the Ornstein-Uhlenbeck Process]\ \\
Let $X\thicksim\caloup_{n,p}(A,B,x_0)$ with a matrix $B\in\calm_p(\bbr)$ such that all eigenvalues of $B$ have a negative real part, i.e. $\mathrm{Re}(\sigma(B))\subseteq(-\infty,0)$.\\
Then the Ornstein-Uhlenbeck process $X$ has a stationary limiting distribution, that is
\[ \caln_{n,p}\left(0,I_n \otimes \cala^{-1}(A^\rmt A) \right) \]
\label{statd_oup}\end{Theorem}

\begin{proof}
Observe that $\mathrm{Re}(\sigma(B))\subseteq(-\infty,0)$ implies $0\notin-\sigma(B)-\sigma(B)$ and we can use the foregoing Theorem. With Theorem \ref{theorem_cfnd} we get for $t\in(0,\infty)$ that $X_t$ given $x_0$ has characteristic function
\beqq \wh{P^{X_t}}(Z)=\etr\left(iZ^\rmt x_0 e^{Bt}-\frac{1}{2}Z^\rmt Z \Psi_t\right) \eeqq
with
\beqq \Psi_t=\cala^{-1}(A^\rmt A)-\cala^{-1}(e^{B^\rmt t}A^\rmt Ae^{Bt}) \label{psi}\eeqq

Now, we write the matrix $B$ in its \emph{Jordan normal form}:\\
There exists a matrix $T\in GL(p)$ such that
\[ B=T \begin{pmatrix} J_1 & 0 & 0 \\ 0 & \ddots & 0 \\ 0 & 0 & J_m \end{pmatrix} T^{-1} \]
with block matrices, called Jordan blocks, $J_i\in\calm_{l_i}(\bbr)$ such that $l_1+\ldots+l_m=p$.
With the calculation rules for the matrix exponential we get
\beqq \exp(Bt)=T\exp\left(\begin{pmatrix} J_1 t & 0 & 0 \\ 0 & \ddots & 0 \\ 0 & 0 & J_m t\end{pmatrix}\right) T^{-1}=T \diag(e^{J_1t},\ldots,e^{J_mt}) T^{-1} \label{jnf}\eeqq

Without loss of generality, consider only the first Jordan block 
\[ J:=J_1= \begin{pmatrix}
\lambda & 1 & 0 & \ldots & 0 \\ 0 & \lambda & 1& \ddots & \vdots \\ \vdots & \ddots & \ddots & \ddots & 0 \\
\vdots &  & \ddots & \ddots & 1 \\ 0 & \ldots & \ldots & 0 & \lambda 
\end{pmatrix}\in\calm_l(\bbr) \]
with $l:=l_1$, $\lambda=a+bi$ an eigenvalue of $B$ with $a<0$.\\
We can separate $J$ into a sum of two matrices, $J=D+N$, with a diagonal matrix $D=\diag(\lambda,\ldots,\lambda)\in\calm_l(\bbr)$ and a matrix
\[ N=\begin{pmatrix}
0& 1 & 0 & \ldots & 0 \\ 0 & 0 & 1& \ddots & \vdots \\ \vdots & \ddots & \ddots & \ddots & 0 \\
\vdots &  & \ddots & \ddots & 1 \\ 0 & \ldots & \ldots & 0 & 0 
\end{pmatrix}\in\calm_l(\bbr) \]
Because $D$ and $N$ commute, $DN=ND$, we have
\[ \exp(Jt)=\exp(Dt+Nt)=\exp(Dt)\exp(Nt) \]
and thus
\[ \exp(Jt)=\diag(e^{\lambda t},\ldots,e^{\lambda t})\exp(Nt)=e^{\lambda t}\exp(Nt)=e^{ibt}e^{at}\exp(Nt) \]
The matrix $N$ is nilpotent, as $N^l=0$. Hence, $\exp(Nt)$ is the finite sum
\[ \exp(Nt)=\sum_{k=0}^{l-1} \frac{1}{k!}N^k=\begin{pmatrix}
1 & t & \ldots & \ldots & \frac{t^{l-1}}{(l-1)!} \\ 0 & 1 & t& \ddots & \vdots \\
\vdots & \ddots & \ddots & \ddots & \vdots \\ \vdots &  & \ddots & \ddots & t \\ 0 & \ldots & \ldots & 0 & 1
\end{pmatrix}\in\calm_l(\bbr) \]
For $a<0$ and $t\rightarrow\infty$ we know that $e^{at}p_t$ converges to zero for every polynomial $p_t$ in $t$. Hence, $\lim_{t\rightarrow\infty}e^{at}\exp(Nt)=0$. As $e^{ibt}$ always has absolute value one, we also have that $\lim_{t\rightarrow\infty}\exp(Jt)=0$. With (\ref{jnf}) it is now obvious that
\[ \lim_{t\rightarrow\infty}\exp(Bt)=0 \]

The next step is to observe that $\cals_p$ is a finite dimensional and normed space with the norm induced by the inner product trace. Hence, the linear operator $\cala^{-1}:\cals_p\rightarrow\cals_p$ is continuous and with the above can conclude that
$\Psi_t$ from (\ref{psi}) converges pointwise to
\[ \lim_{\tto}\Psi_t=\cala^{-1}(A^\rmt A) \]
and thus
\beqq  
\lim_{\tto}\wh{P^{X_t}}(Z)=\etr\left(-\frac{1}{2}Z^\rmt Z\, \cala^{-1}(A^\rmt A)\right):=f(Z)
\eeqq
Clearly, f is continuous at $Z=0$.
With Lévy's Continuity Theorem we know that there exists a probability measure $\mu$ on $\calm_p(\bbr)$ such that $f(Z)=\wh{\mu}(Z)$, and $P^{X_t}\limw\mu$.\\
The function $f$ is the characteristic function of a normal distributed random matrix. That implies by Theorem \ref{theorem_cfnd} that $X_t$ converges in distribution to a random matrix with normal distribution $\mu$,

\beqq \mu=\caln_{n,p}(0,I_n \otimes \cala^{-1}(A^\rmt A) ) \label{ou_limitd}\eeqq

\noindent and the limit is independent of the initial value $x_0$. Thus, and by the fact that $X$ is a Markov process according to Theorem \ref{markovsolution}, the limit distribution in (\ref{ou_limitd}) is a stationary distribution.
\end{proof}

\begin{Corollary}
Let $X\thicksim\caloup_{n,p}(A,B,x_0)$ with a matrix $B\in\calm_p(\bbr)$ such that $B\in\cals_p^-$ and $A^\rmt A$ and $B$ commute. Then the stationary distribution of $X$ is given by
\[ \caln_{n,p}\left(0,I_n \otimes - \frac{1}{2} A^\rmt A B^{-1} \right) \]
\end{Corollary}

\begin{proof}
\[ \cala^{-1}(A^\rmt A)=- \frac{1}{2} A^\rmt A B^{-1} \Leftrightarrow A^\rmt A = \cala(- \frac{1}{2} A^\rmt A B^{-1})=\frac{1}{2}(B A^\rmt A B^{-1} + A^\rmt A)= A^\rmt A\]
\end{proof}

\chapter{Wishart Processes}

In this chapter the main work on the theory of Wishart processes will be done. As an introduction we begin with the one dimensional case, called the square Bessel process. Afterwards, we focus on the general Wishart process and give a theorem about the existence and uniqueness of Wishart processes at the end of section 4.2. In section 4.3, we show that some Wishart processes can be expressed as matrix squares of the matrix variate Ornstein-Uhlenbeck process from section 3.5. Eventually, in section 4.4 we give an algorithm to simulate Wishart processes by using an Euler approximation.

\section{The one dimensional case:\\ Square Bessel Processes}

The square Bessel process is the one-dimensional case of the Wishart process - in fact, even a special case of the one-dimensional Wishart process. We consider it separately, because it is much more elementary than for higher dimensions.

\begin{Definition}[Square Bessel Process]
Let $\al\geq 0$, $\beta$ be a one-dimensional Brownian motion and $x_0\geq 0$. A strong solution of
\beam dX_t=2\sqrt{X_t}\,d\beta_t+\al\,dt, \quad X_0=x_0 \label{eq_besq}\eeam
is called a square Bessel process with parameter $\al$ and denoted by $X\thicksim\rmbesq(\al,x_0)$.
\end{Definition}

\begin{Remark} Let $B$ denote a Brownian motion on $\bbr^n$ (that is an $\bbr^n$-vector of n \\ independent one-dimensional Brownian motions). The Bessel process is then the Euclidean norm of $B$, that is $\sqrt{B^\rmt B}$. Hence, the process $X$ with $X_t=B_t^\rmt B_t$ is often called square Bessel process in the literature. Using Theorem \ref{lemma_trbm} and Theorem \ref{partialintegration} one can show that $X\thicksim\rmbesq(n,x_0)$. Hence, using (\ref{eq_besq}) for the definition  of the square Bessel process is more general as it allows $n=\al$ to be any real, non-negative number.
\end{Remark}

Before we give a comprehensive theorem about existence, uniqueness and non-negativity of the square Bessel process, we state the following results for one-dimensional stochastic differential equations:

\begin{Theorem}(Pathwise Uniqueness, cf. \citet[Theorem 1]{yamada1})
Let 
\beam dX_t=\sigma(X_t)\,d\beta_t + b(X_t)\,dt \label{eq_1dsde}\eeam
be a one-dimensional stochastic differential equation where $b,\sigma:\bbr\rightarrow\bbr$ are continuous functions and $\beta$ a one-dimensional Brownian motion. Suppose there exist increasing functions $\rho,\kappa:(0,\infty)\rightarrow (0,\infty)$ such that
\beam
|\sigma(\xi)-\sigma(\eta)|\leq\rho(|\xi-\eta|) \quad\fa \xi,\eta\in\bbr \nonumber \\
\textnormal{ with } \int_0^1 \rho^{-2}(u)\,du=\infty
\label{eq_rho}\eeam 
and
\beam
|b(\xi)-b(\eta)|\leq\kappa(|\xi-\eta|) \quad\fa \xi,\eta\in\bbr \nonumber \\
\textnormal{ with } \int_0^1 \kappa^{-1}(u)\,du=\infty
\eeam 
Then pathwise uniqueness holds for (\ref{eq_1dsde}).
\label{theorem_pu1}\end{Theorem}

\begin{Theorem}(Comparison Theorem, \citet[Chapter IX, Theorem 3.7]{revuz})
Consider two stochastic differential equations
\beao dX^1_t=\sigma(X_t^1)\,d\beta_t + b^1(X_t^1)\,dt,\quad dX^2_t=\sigma(X_t^2)\,d\beta_t + b^2(X_t^2)\,dt \eeao
that both fulfill pathwise uniqueness and let $b^1$,$b^2$ be two bounded Borel functions such that $b^1\geq b^2$ everywhere and one of them is globally Lipschitz. If $(X^1,\beta)$ is a solution of the first, and $(X^2,\beta)$ a solution of the second stochastic differential equation ($X^1$,$X^2$ defined on the same probability space), w.r.t the same Brownian motion $\beta$  and if $X^1_0\geq X^2_0$ a.s., then
\beqq P[X^1_t\geq X^2_t \textit{ for all } t\in\bbr_+]=1 \eeqq
\end{Theorem}

\begin{Theorem}[Existence, Uniqueness and Non-Negativity of the Square Bessel Process]
For $\al\geq 0$ and $x_0\geq 0$ there exists an unique strong and non-negative solution $X$ of
\beqq X_t=x_0+2\int_0^t\sqrt{X_s}\,d\beta_s+\al t \label{besq2}\eeqq
on the entire interval $[0,\infty)$.\\
Moreover, if $\al\geq 2$ and $x_0>0$ the solution $X$ is positive a.s. on $[0,\infty)$. 
\label{theorem_expupbesq}\end{Theorem}

\begin{proof}
We show with Theorem \ref{theorem_pu1} that pathwise uniqueness holds for the stochastic differential equation
\beam dX_t=2\sqrt{|X_t|}\,d\beta_t+\al\,dt, \quad X_0=x_0 \label{besq_with_abs}\eeam
like in \citet[p. 439]{revuz}. Since $|\sqrt{z}-\sqrt{z'}|\leq\sqrt{|z-z'|}$ for all $z,z'\geq 0$, we have that
\[ 2\, \left| \sqrt{|\xi|}-\sqrt{|\eta|} \right| \leq 2 \sqrt{\left|\,|\xi|-|\eta|\,\right|} \leq 2 \sqrt{|\xi-\eta|}=\rho(|\xi-\eta|) \quad\fa \xi,\eta\in\bbr \]
with $\rho(u)=2\sqrt{u}$. Clearly, 
\beqq \int_0^1 \rho^{-2}(u)\,du= \frac{1}{4} \int_0^1 \frac{1}{u}\,du=\infty \eeqq
With Remark \ref{remark_pu} we conclude that there exists a unique strong solution for (\ref{besq_with_abs}).\\
Next, we show that our solution never becomes negative: First, consider the case $x^1_0=0$ and $\al^1=0$. Obviously, $X^1\equiv 0$ is our unique strong solution in this case.
Secondly, consider an arbitrary $x^2_0\geq 0$ and $\al^2\geq 0$. From the above, we have a unique strong solution $X^2$. The comparison theorem implies that 
\[ P[X^2_t\geq X^1_t \textit{ for all } t\in\bbr_+]=P[X^2_t\geq 0 \textit{ for all } t\in\bbr_+]=1 \]
and thus $X^2_t$ is non-negative for all $t$ almost surely. Hence, we can discard the $|.|$ in (\ref{besq_with_abs}) and $X^2$ also is the unique strong solution of (\ref{besq2}).\\
Finally, \citet[Proposition 1.5, Chapter XI]{revuz} has shown that for $\al\geq 2$, the set $\{0\}$ is polar, i.e. $P(\inf\{s:\,X_s=0\}<\infty)=0$ for all initial values $x_0>0$. Hence, in this case the unique strong solution $X$ is positive for all $t\in\bbr_+$.
\end{proof}

\begin{Remark}
For $\al\geq 2$, one could also use McKean's argument (Theorem \ref{theorem_mckean}) to show that the square Bessel process is positive, analogously to the Proof of Theorem \ref{theorem_exwp3}.
\end{Remark}

\section{General Wishart Processes:\\ Definition and Existence Theorems}

\begin{Definition}[Wishart Processes]
Let $B$ be a $p\times p$-dimensional Brownian motion, $B\thicksim\calbm_p$, $Q\in\calm_p(\bbr)$ and $K\in\calm_p(\bbr)$ be arbitrary matrices, $s_0\in\ov{\cals_p^+}$ the initial value and $\al\geq0$ a non-negative number. Then we call the stochastic differential equation
\beqq
dS_t=\sqrt{S_t}\,dB_t\,Q+Q^\rmt dB_t^\rmt \sqrt{S_t}+(S_t K+K^\rmt S_t+ \al Q^\rmt Q)\,dt,\quad S_0=s_0
\label{eq_wp}\eeqq
the \emph{Wishart SDE}.\\
A strong solution $S$ of (\ref{eq_wp}) in $\ov{\cals_p^+}$ is said to be a ($p\times p$-dimensional) \emph{Wishart process} with parameters $Q,K,\al,s_0$, written $S\thicksim\calwp_p(Q,K,\al,s_0)$.
\label{def_wp}\end{Definition}

To understand (\ref{eq_wp}) intuitively, it may help to have a look at a approximation of the form
\[ S_{t+h}\approx S_t+\sqrt{S_t}\int_t^{t+h}\,dB_s\,Q+\int_t^{t+h} Q^\rmt dB_s^\rmt\sqrt{S_t}+(S_t K+K^\rmt S_t+ \al Q^\rmt Q)h \]
Here we can see that $Q$ controls the covariance of our normal distributed fluctuations, and that this fluctuations are proportional to the square root of our process:
\[ \sqrt{S_t}\int_t^{t+h} dB_s\,Q | S_t \thicksim \sqrt{S_t}\cdot \caln_{p,p}(0,I_p\otimes Q^\rmt Q h )  \]
Hence, the fluctuations decrease quickly if our process goes to zero.\\
As $\al Q^\rmt Q$ is positive semidefinite, the parameter $\al$ determines how large the drift away from zero is.\\
\noindent Just like in the one-dimensional case, the solution of (\ref{eq_wp}) has a \emph{mean reverting} property. To understand that, we consider the deterministic equivalent of (\ref{eq_wp}), that is the ordinary differential equation
\[ \frac{dS_t}{dt}=S_t K+K^\rmt S_t+ \al Q^\rmt Q, \quad S_0=s_0 \]
To simplify the notation, we define the operator $\calc:\cals_p\rightarrow\cals_p,X\mapsto XK+K^\rmt X$ and set $M=\al Q^\rmt Q$. Then
\beqq \frac{dS_t}{dt}=\calc S_t + M, \quad S_0=s_0 \label{eq_detwis}\eeqq
From the theory of ordinary linear differential equations (see \citet[p. 83]{rep}, for example), we know that the solution of (\ref{eq_detwis}) is given by
\beqq S_t=e^{\calc t}\left(s_0+\int_0^t e^{-\calc u}\,du M\right) \label{sol_detwis}\eeqq
Indeed, by partial integration follows
\[ \frac{dS_t}{dt}=\calc\underbrace{\left(e^{\calc t}\left(s_0+\int_0^t e^{-\calc u}\,du M\right)\right)}_{=S_t}+\underbrace{e^{\calc t}e^{-\calc t}}_{=I_p}M \]
Evaluating the integral in (\ref{sol_detwis}) gives
\beqq S_t=e^{\calc t}( (s_0 + \calc^{-1} M) - \calc^{-1} M \label{sol_detwis2}\eeqq
If the eigenvalues of $K$ only have negative real parts, $Re(\sigma(K))\subseteq(-\infty,0)$, then, because of $\sigma(\calc)=\sigma(K)+\sigma(K)$, we also have $Re(\sigma(\calc))\subseteq(-\infty,0)$. From the proof of Theorem \ref{statd_oup} we get that $\lim_{t\rightarrow\infty}e^{\calc t}=0$ and thus from (\ref{sol_detwis2})
\[ \lim_{t\rightarrow\infty}S_t= - \calc^{-1} M \]
Hence, the deterministic solution of (\ref{eq_detwis}) converges to $- \calc^{-1} M $. Thus, the stochastic solution of (\ref{eq_wp}) will fluctuate around $- \calc^{-1} M $.

\vspace{0.5cm}
Considering the fact that the fluctuations decrease quickly if our processes goes to zero and that we have a non-negative definite drift $\al Q^\rmt Q$ away from zero one may think that this process never leaves $\ov{\cals_p^+}$. As we have shown in the last section, for the one dimensional case (square Bessel process) we have a unique strong solution for all $t$ that never becomes negative.
Unlike in this case, we can not show for $p\geq 2$ that there still exists a unique strong solution after the process
hits the boundary of $\cals_p^+$ (that means $S_t\in\ov{\cals_p^+}\backslash\cals_p^+$) the first time. We are only able to give sufficient conditions such that the process stays in the set of all positive definite matrices, and then we have a unique strong solution for all $t$.
\vspace{0.5cm}\\
One reason why we cannot transfer the proof of Theorem \ref{theorem_expupbesq} to the matrix variate case is that  Theorem \ref{theorem_pu1} cannot be generalized to matrix variate stochastic differential equations. In the matrix variate version of Theorem \ref{theorem_pu1} as stated in \citet[Theorem1]{yamada2}, the constraint (\ref{eq_rho}) turns to
\beam
&& \int_{\{U\in\ov{\cals_p^+}:||U||\leq1\}} \rho^{-2}(U)U\,dU=\infty \label{eq_rho2}\\
&& U\mapsto \rho^{2}(U)U^{-1} \textnormal{ is concave}
\eeam
but
\[ \int_{\{U\in\ov{\cals_p^+}:||U||\leq1\}} \rho^{-2}(U)U\,dU=\int_{\{U\in\ov{\cals_p^+}:||U||\leq1\}} I_p\,dU<\infty \]
for $\rho(U)=\sqrt{U}$. Hence the theorem can not be applied to our case.\\
Furthermore, \citet[Remark 2]{yamada2} have also shown that (\ref{eq_rho2}) is, for $p\geq 3$, nearly best possible in the sense that, if $\int_{U\in\ov{\cals_p^+}:||U||\leq 1} \rho^{-2}(U)U\,du<\infty$ and $\rho$ is subadditive then,
a stochastic differential equation can be constructed that has two solutions, and thus pathwise uniqueness cannot hold. But the fact the we cannot use \citet[Theorem1]{yamada2} does not give any evidence whether pathwise uniqueness holds for the Wishart SDE or not.
\vspace{1cm}\\
Before we can begin our mathematical analysis of (\ref{eq_wp}), we need a few auxiliary results:

\begin{Lemma}[Quadratic Variation of the Wishart Process]
Let $S\thicksim\calwp_p(Q,K,\al,s_0)$. Then
\[ d[S_{ij},S_{kl}]_t= S_{t,ik} (Q^\rmt Q)_{jl} \,dt +  S_{t,il} (Q^\rmt Q)_{jk} \,dt +  S_{t,jk} (Q^\rmt Q)_{il} \,dt +  S_{t,jl} (Q^\rmt Q)_{ik} \,dt \]
as long as $S$ exists.
\label{lemma_qvwp}\end{Lemma}

\begin{proof}
Define $H_t:=\sqrt{S_t}dB_tQ+Q^\rmt dB_t^\rmt \sqrt{S_t}$. Then $[S_{ij},S_{kl}]=[H_{ij},H_{kl}]$.
\[ H_{t,ij}=\sum_{m,n} (\sqrt{S_t})_{in} dB_{t,nm} Q_{mj} + Q_{mi} dB_{t,nm} (\sqrt{S_t})_{nj} \]
The first summand of $d[H_{ij},H_{kl}]_t$ is equal to
\[ \sum_{m,n} (\sqrt{S_t})_{in} (\sqrt{S_t})_{kn} Q_{mj} Q_{ml} \,dt \]
Because $S_t$ is symmetric this term can be simplified to
\[ S_{t,ik} (Q^\rmt Q)_{jl} \,dt \]
The other three summands can be evaluated in the same way.
\end{proof}

\begin{Lemma}
Let $S\in\cals_p^+$ be a stochastic process, $B \thicksim \calbm_p$ and $h:\calm_p(\bbr)\rightarrow\calm_p(\bbr)$. Then there exists a one dimensional Brownian motion $\beta^h$ such that
\[ \tr\left(\int_0^t h(S_u)\,dB_u\right)=\int_0^t \sqrt{\tr(h(S_u)^\rmt h(S_u))}\,d\beta_u^h \]
\label{lemma_trbm}\end{Lemma}

\begin{proof}
Define 
\[ \beta_T^h:=\sum_{i,n=1}^p \int_0^T \frac{h(S_t)_{in}}{\sqrt{\tr(h(S_t)^\rmt h(S_t))}}\,dB_{t,ni} \]
Observe that the denominator is zero if and only if the numerator is zero, so we make the convention $\frac{0}{0}:=1$. Then we have by definition
\[ \tr(h(S_t)\,dB_t) = \sum_{i,n=1}^p  h(S_t)_{in}\,dB_{t,ni} = \sqrt{\tr(h(S_t)^\rmt h(S_t))}\,d\beta_t^h \]
For every $i,n=1,\ldots,p$ we have
\beao
\int_0^T \left(\frac{h(S)_{in}}{\sqrt{\tr(h(S)^\rmt h(S))}}\right)^2\,dt = \int_0^T \frac{h(S)_{in}^2}{\tr(h(S)^\rmt h(S))}\,dt \\
\leq  \int_0^T \sum_{i,n=1}^p \frac{h(S)_{in}^2}{\tr(h(S)^\rmt h(S))}\,dt = T < \infty
\eeao
and therefore $\beta^h$ is a sum of continuous local martingales and thus a continuous local martingale itself. Furthermore
\[ [\beta^h,\beta^h]_T= \int_0^T \sum_{i,n} \frac{h(S_t)_{in}^2}{\tr(h(S_t)^\rmt h(S_t))}\,dt = \int_0^T dt=T\]
and with Lévy's Theorem the proof is complete.
\end{proof}

\begin{Lemma}
The matrix variate square root function is locally Lipschitz on $\cals_p^+$.
\end{Lemma}

\begin{proof} See \citet{stelzer}. \end{proof}

\begin{Theorem}[Existence and Uniqueness of the Wishart Process I]\ \\
For every initial value $s_0\in\cals_p^+$ there exists a unique strong solution $S$ of the Wishart SDE (\ref{eq_wp}) in the cone $\cals_p^+$ of all positive definite matrices up to the stopping time
\[ T=\inf\{s:\,\det(S_s)=0\}>0 \textit{ a.s.} \]
\label{theorem_exwp1}\end{Theorem}

\begin{proof}
We check that the assumptions of Theorem \ref{theorem_exsde} are fulfilled. Observe, that the set $U=\cals_p^+$ is open and that there exists a sequence of convex closed subsets $(U_n)_{n\in\bbn}$ of $U$ that are increasing w.r.t. $\subseteq$ and $\bigcup_{n\in\bbn}U_n=U$. Indeed, observe that  
the function that maps every Matrix $M\in\cals_p^+$ to its smallest eigenvalue, \[ \lambda_{\min}:\cals_p^+\rightarrow(0,\infty),M\mapsto\lambda_{\min}(M)=\min_{||v||=1}v^\rmt M v \]
is continuously differentiable, see \citet[Lemma 5.1]{num}. Hence, 
\[ U_n:=\{M\in\cals_p^+:\,\lambda_{\min}(M)\geq\frac{1}{n}\}=\lambda_{\min}^{-1}\left (\underbrace{\left[\frac{1}{n},\infty\right)}_{closed} \right) \]
is a closed set; and also convex, as for all $ M_1,M_2\in U_n, \al\in[0,1] $:
\beao
\lambda_{\min}(\al M_1 + (1 - \al) M_2) & = & \min_{||v||=1}v^\rmt (\al M_1 + (1 - \al) M_2) v \\
 & \geq & \al \min_{||v||=1}v^\rmt M_1 v + (1-\al) \min_{||v||=1}v^\rmt M_2 v \\ & \geq & 
 \al \frac{1}{n} + (1-\al) \frac{1}{n} = \frac{1}{n}
\eeao

Next, we define for $S\in\cals_p^+$ the linear operator by
\[ \calz_{S}=\calz(S):\calm_p(\bbr)\rightarrow\cals_p,X\mapsto \sqrt{S}XQ+Q^\rmt X^\rmt \sqrt{S} \]
and as before
\[ \calc:\cals_p\rightarrow\cals_p,X\mapsto XK + K^\rmt X \]
Then we can write (\ref{eq_wp}) in the form
\beao S_t & = & \int_0^t \calz_{S_u} dB_u + \int_0^t (\calc S_u + \al Q^\rmt Q)\,du \\
& = & \int_0^t \underbrace{\begin{pmatrix} \calz(S_u) \\ \calc S_u + \al Q^\rmt Q \end{pmatrix} ^\rmt}_{:=f(S_u)}\,
			d\underbrace{\begin{pmatrix} B_u \\ u I_p \end{pmatrix}}_{=:Z_u} \\
& = & \int_0^t f(S_u)\,dZ_u 			
\eeao
\noindent where \citet{met} give a formal justification why we can integrate w.r.t $\calz_{S_u}dB_u$. Obviously, $Z$ is a continuous semimartingale.
We still have to show that the function $f$ is locally Lipschitz. For any norm given on $\calm_{p}(\bbr)$, we define a norm on $\calm_{2p,p}(\bbr)$ by
\[ ||(X,Y)^\rmt||_{\calm_{2p,p}(\bbr)}=||X||_{\calm_{p}(\bbr)}+||Y||_{\calm_{p}(\bbr)} \quad\fa X,Y\in\calm_{p}(\bbr) \]
but discard the subscripts as it should be obvious which norm to use. Then, for all $S,R\in\cals_p^+$ we have
\[ ||f(S)-f(R)||=||\calz_S - \calz_R|| + ||\calc(S-R)|| \]
Because $\calc$ is a bounded linear operator ($\dim\cals_p<\infty$), we have
\[ ||\calc(S-R)||\leq ||\calc|| \, ||S-R|| \quad\fa S,R\in\cals_p^+ \textnormal{ with } ||\calc||<\infty \]
Let be $S,R\in\cals_p^+$. For $X\in\calm_p(\bbr)$ arbitrary we have
\beao ||(\calz_S - \calz_R) X|| & = & ||(\sqrt{S}-\sqrt{R})XQ+Q^\rmt X^\rmt (\sqrt{S}-\sqrt{R})|| \\
& \leq & 2 ||\sqrt{S}-\sqrt{R}|| \, ||X|| \, ||Q|| \\
& \Rightarrow & ||\calz_S-\calz_R||=\sup_{X\in\calm_p(\bbr)\backslash\{0\}} \frac{||(\calz_S - \calz_R) X||}{||X||} \leq 2 ||\sqrt{S}-\sqrt{R}|| \, ||Q||
\eeao
Let now be $Y\in\cals_p^+$. As the matrix variate square root function is locally Lipschitz, there exists an open neighbourhood $\calu(Y)$ of $Y$ and a constant $C(Y)$, such that for all $S,R\in\calu(Y)$
\[ ||\sqrt{S}-\sqrt{R}|| \leq C(Y) \, ||S-R|| \]
Hence, we have for all $S,R\in\calu(Y)$
\[ ||\calz_S-\calz_R|| \leq \underbrace{2 \, C(Y) \, ||Q||}_{=:C'(Y)} ||S-R|| = C'(Y) \, ||S-R|| \]
i.e. $\calz$ is locally Lipschitz on $\cals_p^+$. At all $f$ is locally Lipschitz because
\beqq ||f(S)-f(R)||\leq \underbrace{( C'(Y) + ||\calc|| )}_{=:K} \, ||S-R|| = K \,||S-R|| \quad\fa S,R\in\calu(Y) \label{f_loclip}\eeqq

Hence, we know that there exists a non-zero stopping time $T>0$ such that there exists a unique $\cals_p^+$-valued strong solution $S$ of (\ref{eq_wp}) for $t\in[0,T)$. If $T<\infty$, $S_T$ either hits the boundary of $\cals_p^+$ or explodes. We show that the latter cannot happen. Fix $R\in\calu(Y)$ and set $S=Y$. Then we from (\ref{f_loclip})
\beao ||f(Y)-f(R)|| &\leq&  K\, (||Y||+||R||) \\
\Rightarrow ||f(Y)-f(R)||^2 &\leq& K^2\, (||Y||^2+2||Y||\,||R||+||R||^2) \leq L\, (1+||Y||+||Y||^2) \eeao
with $L=\max\{K^2||R||^2,2 K^2 ||R||, K^2\}$. Because of $||Y||\leq 1+ ||Y||^2$
\[ ||f(Y)-f(R)||^2 \leq {2} L \,(1+||Y||^2) \]
Hence, $Y\mapsto f(Y)-f(R)$ satisfies the linear growth condition, and so does $f:Y\mapsto f(Y)$.

Thus, if $T<\infty$, we know that $S_T$ hits the boundary of $\cals_p^+$ the first time, i.e. $S_T\in\ov{\cals_p^+}$, but $S_T\notin\cals_p^+$, and $S_t\in\cals_p^+$ for all $t<T$. Hence
\[ T=\inf\{u:\,S_u\notin\cals_p^+\} \]
or
\[  T=\inf\{s:\,\det(S_s)=0\} \]
\end{proof}

In other words: A unique strong solution $S$ of (\ref{eq_wp}) exists as long as $S$ stays in the interior of $\ov{\cals_p^+}$.\\ As a consequence, in order to show that there exists a unique strong solution $S$ of (\ref{eq_wp}) in $\cals_p^+$ on the entire interval $[0,\infty)$ we only need to show that $S_t\in\cals_p^+$ for all $t\in\bbr_+$.\\
Hence, the next step is to give sufficient conditions that guarantee that $S_t\in\cals_p^+$ for all $t\in\bbr_+$. First, we do this for the case with a zero drift, $K=0$. Later, we can generalize our results for $K\neq 0$ using a Girsanov transformation.

\begin{Theorem}
For every initial value $s_0\in\cals_p^+$ there exists a stopping time $T>0$ and a unique strong solution of the Wishart SDE (\ref{eq_wp}) on $[0,T)$. Suppose $T<\infty$. Then there exists a unique strong solution $(S_t)_{t\in[0,T]}$ of the Wishart SDE (\ref{eq_wp}) on $[0,T)$ such that $S_t\in\cals_p^+$ for all $t\in[0,T)$ and  $S_T\in\partial\cals_p^+=\ov{\cals_p^+}\backslash\cals_p^+$. Then, for every $x\in\bbr^p$ with $x^\rmt Q^\rmt Q x=1$ the process $(x^\rmt S_t x)_{t\in[0,T]}$ is a square Bessel process with parameter $\al$ and initial value $x^\rmt s_0 x$, $(x^\rmt S_t x)_{t\in[0,T]}\thicksim\rmbesq(\al,x^\rmt s_0 x)$.
Moreover, if $\al\geq 2$ then the process $(x^\rmt S_t x)_{t\in[0,T]}$ remains positive almost surely. If furthermore $Q\in GL(p)$ then:
\[ \textit{For every } y\in\bbr^p,y\neq 0, \textit{ it holds that } y^\rmt S_t y>0 \textit{ for all $t\in[0,T]$ a.s.} \]
\label{theorem_exwp2}\end{Theorem}

\begin{proof}
We get the required solution $(S_t)_{t\in[0,T]}$ of (\ref{eq_wp}) by Theorem \ref{theorem_exwp1} if we attach the one point $S_T$ to the solution on $[0,T)$ of (\ref{eq_wp}).\\
Let $x\in\bbr^p$ be an arbitrary vector with $x^\rmt Q^\rmt Q x = 1$. The matrix of all partial derivatives of the map $\calm_p(\bbr)\ni S\mapsto x^\rmt S x\in\bbr$ is equal to
\[ D(x^\rmt S x)=(x_ix_j)_{i,j}=xx^\rmt \]
and hence all second derivatives are zero. \\
With Itô's formula (Theorem \ref{theorem_ito}) we get
\beao
d(x^\rmt S_t x) & = & \tr(xx^\rmt \,dS_t) \\
& = &  \tr(Qxx^\rmt \sqrt{S_t}\,dB_t)+\tr(\sqrt{S_t}xx^\rmt Q^\rmt \,dB_t^\rmt)+\tr(xx^\rmt \al Q^\rmt Q \,dt) \\
& = &  2\,\tr(Qxx^\rmt \sqrt{S_t}\,dB_t)+\tr(xx^\rmt \al Q^\rmt Q \,dt) \\
& = &  2 \sqrt{\tr(S_txx^\rmt Q^\rmt Q xx^\rmt)}\,d\beta_t+\al\tr(x^\rmt Q^\rmt Q x)\,dt \\
& = &  2 \sqrt{\tr(S_t x x^\rmt)}\,d\beta_t+\al\,dt \\
& = &  2 \sqrt{x^\rmt S_t x}\,d\beta_t+\al\,dt
\eeao
where we used Lemma \ref{lemma_trbm}. Hence, $(x^\rmt S_t x)_{t\in[0,T]} \thicksim \rmbesq(\al,x^\rmt s_0 x)$.\\
If $\al\geq2$ we know from Theorem \ref{theorem_expupbesq} that the process $(x^\rmt S_t x)_{t\in[0,T]}$ is strictly positive for all $t\in[0,T]$ a.s., because the initial value is positive, $x^\rmt s_0 x>0$.\\
Now suppose that $Q\in GL(p)$. Let $y\in\bbr^p,y\neq 0$, then for $x:=(y^\rmt Q^\rmt Q y)^{-\frac{1}{2}}y$ it holds that $x^\rmt Q^\rmt Q x = 1$ and by the above that $x^\rmt S_t x>0$ for all $t\in[0,T]$ a.s. Thus, we also have $y^\rmt S_t y>0$ for all $t\in[0,T]$ a.s.
\end{proof}

First it may sound surprisingly that for fixed $y\neq 0$, the process $(y^\rmt S_t y)_{t\in[0,T]}$ is almost surely positive at $T$, $y^\rmt S_T y > 0$ a.s., even though that there exists an $z\in\bbr^p,z\neq 0$, such that $z^\rmt S_T z=0$ a.s. But this just tells us, that it is `unlikely' to find such a vector $z$ in advance. Before we continue with the next theorem, we state

\begin{Theorem}
Let $S\thicksim\calwp_p(Q,0,\al,s_0)$. Then, for $\xi\neq 0$
\beam
& & d(\det(S_t))=2\det(S_t)\sqrt{\tr(Q^\rmt Q S_t^{-1})}d\beta_t+\det(S_t)(\al+1-p)\tr(Q^\rmt Q S_t^{-1})\,dt\label{eq_1}\\
& & d(\det(S_t)^\xi)=2\xi\det(S_t)^\xi\left[\sqrt{\tr(Q^\rmt Q S_t^{-1})}\,d\beta_t+\tr(Q^\rmt Q S_t^{-1})(\frac{\al-1-p}{2}+\xi)dt\right] \nonumber\\ \label{eq_2}\\
& & d(\ln(\det(S_t)))=2 \sqrt{\tr(Q^\rmt Q S_t^{-1})}\,d\beta_t+(\al-p-1)\tr(Q^\rmt Q S_t^{-1})\,dt \nonumber\\ 
& & d(\ln(\det(S_t)))=2\sqrt{\tr(Q^\rmt Q S_t^{-1})}\,d\beta_t \textit{ for }\al=p+1 \label{eq_4}
\eeam
for $t\in[0,T)$ with $T=\inf\{s:\,\det(S_s)=0\}$, where $\beta$ is a one-dimensional Brownian motion.
\end{Theorem}

These results can also be found in \citet[p. 747]{bru} without proof.

\begin{proof}
We consider the SDE
\[ dS_t=\sqrt{S_t}dB_tQ+Q^\rmt dB_t^\rmt \sqrt{S_t}+ \al Q^\rmt Q dt,\quad S_0=s_0 \]
First, we prove Equation (\ref{eq_1}): \\
According to Itô's formula, we have 
\[ d(\det(S_t))=\tr(D(\det(S_t))\,dS_t)+ \frac{1}{2} \sum_{i,j,k,l=1}^p \frac{\partial^2}{\partial S_{t,ij}\,\partial S_{t,kl}} \det(S_t) \, d[S_{ij},S_{kl}]_t \]
Using Lemma \ref{lemma_calcrulesdet} we get
\beao
\tr(D(\det(S))\,dS_t) & = &  \det(S_t)\tr(S_t^{-1}\,dS_t) \\
& = &  \det(S_t)\tr(S_t^{-1}(\sqrt{S_t}dB_tQ+Q^\rmt dB_t^\rmt \sqrt{S_t}+ \al Q^\rmt Q dt)) \\
& = &  \det(S_t)[\tr(Q S_t^{-\frac{1}{2}} \,dB_t)+\tr(S_t^{-\frac{1}{2}} Q^\rmt \,dB_t^\rmt)+\al\tr(Q^\rmt Q S_t^{-1})] \\
& = &  \det(S_t)[2\sqrt{\tr(Q^\rmt Q S_t^{-1})}\,d\beta_t+\al\tr(Q^\rmt Q S_t^{-1})]
\eeao
where we used Lemma \ref{lemma_trbm} in the last equation.
For the second order term, we get
\beao
& & \frac{1}{2} \sum_{i,j,k,l=1}^p \frac{\partial^2}{\partial S_{t,ij}\,\partial S_{t,kl}} \det(S_t) \, d[S_{ij},S_{kl}]_t \\
& = & \frac{1}{2} \sum_{i,j,k,l=1}^p \det(S)[(S_t^{-1})_{kl}(S_t^{-1})_{ij}-(S_t^{-1})_{ik}(S_t^{-1})_{lj}]d[S_{ij},S_{kl}]_t \\
& = &  \frac{1}{2} \sum_{i,j,k,l=1}^p \det(S)[(S_t^{-1})_{kl}(S_t^{-1})_{ij}-(S_t^{-1})_{ik}(S_t^{-1})_{lj}] [S_{t,ik} (Q^\rmt Q)_{jl} \,dt +  S_{t,il} (Q^\rmt Q)_{jk} \,dt \\
& &\hspace{8cm}+  S_{t,jk} (Q^\rmt Q)_{il} \,dt +  S_{t,jl} (Q^\rmt Q)_{ik} \,dt] \\
& = &  \det(S_t)[(1-p)\tr(Q Q^\rmt S_t^{-1})\,dt]
\eeao
where we used Lemma \ref{lemma_qvwp} and Lemma \ref{lemma_calcrulesdet}, again.
At all, we get equation (\ref{eq_1})
\[ d(\det(S_t))=2\det(S_t)\sqrt{\tr(Q^\rmt Q S_t^{-1})}\,d\beta_t+\det(S_t)(\al+1-p)\tr(Q^\rmt Q S_t^{-1})\,dt \]

For Equation (\ref{eq_2}), we observe that
\beam dX_t^\xi=\xi X_t^{\xi-1}\,dX_t+\frac{1}{2}\xi(\xi-1)X_t^{\xi-2}d[X,X]_t \label{eq_xi}\eeam
If we set 
\[X_t:=det(S_t)\]
then (\ref{eq_1}) is equal to
\beqq
dX_t=2X_t\left(\sqrt{\tr(Q^\rmt Q S_t^{-1})}\,d\beta_t+\frac{\al+1-p}{2}\tr(Q^\rmt Q S_t^{-1})\,dt\right)
\label{eqx}\eeqq
and
\[ d[X,X]_t=4X_t^2\tr(Q^\rmt Q S_t^{-1})\,dt \]
If we insert (\ref{eqx}) into (\ref{eq_xi}) we get (\ref{eq_2}):
\beao
dX_t^\xi & = & \xi X_t^{\xi-1}2 X_t\left(\sqrt{\tr(Q^\rmt Q S_t^{-1})}\,d\beta_t+\frac{\al+1-p}{2}\tr(Q^\rmt Q S_t^{-1})\,dt\right)\\
&+&\frac{1}{2}\xi(\xi-1)X_t^{\xi-2}4X_t^2\tr(Q^\rmt Q S_t^{-1})\,dt\\
& = & 2\xi X_t^{\xi}\left[\sqrt{\tr(Q^\rmt Q S_t^{-1})}\,d\beta_t+\frac{\al+1-p}{2}\tr(Q^\rmt Q S_t^{-1})\,dt+(\xi-1)\tr(Q^\rmt Q S_t^{-1})\,dt\right]\\
& = & 2\xi X_t^{\xi}\left[\sqrt{\tr(Q^\rmt Q S_t^{-1})}\,d\beta_t+\tr(Q^\rmt Q S_t^{-1})\left(\frac{\al-1-p}{2}+\xi\right)dt\right]
\eeao

To prove the last equation, observe that
\beam d(\ln(X_t))=X_t^{-1}\,dX_t-\frac{1}{2}X_t^{-2}\,d[X,X]_t \label{eq_ln}\eeam
and insert (\ref{eqx}) into (\ref{eq_ln}):
\beao
d(\ln(X_t)) & = & X_t^{-1}2X_t\left(\sqrt{\tr(Q^\rmt Q S_t^{-1})}\,d\beta_t+\frac{\al+1-p}{2}\tr(Q^\rmt Q S_t^{-1})\,dt\right)\\
& - & \frac{1}{2}X_t^{-2}4X_t^2\tr(Q^\rmt Q S_t^{-1})\,dt \\
& = & 2\left(\sqrt{\tr(Q^\rmt Q S_t^{-1})}\,d\beta_t+\frac{\al+1-p}{2}\tr(Q^\rmt Q S_t^{-1})\,dt\right)-2\tr(Q^\rmt Q S_t^{-1})\,dt \\
& = & 2 \sqrt{\tr(Q^\rmt Q S_t^{-1})}\,d\beta_t+(\al+1-p-2)\tr(Q^\rmt Q S_t^{-1})\,dt \\
& = & 2 \sqrt{\tr(Q^\rmt Q S_t^{-1})}\,d\beta_t+(\al-p-1)\tr(Q^\rmt Q S_t^{-1})\,dt
\eeao
which equals (\ref{eq_4}), if $\al=p+1$.
\end{proof}

\begin{Theorem}[Existence and Uniqueness of the Wishart Process II]\ \\
Let $K=0$, $s_0\in\cals_p^+$ and $\al \geq p+1$. Then there exists a unique strong solution in $\cals_p^+$ of the Wishart SDE (\ref{eq_wp}) on $[0,\infty)$.
\label{theorem_exwp3}\end{Theorem}

\begin{proof}
We consider the SDE
\[ dS_t=\sqrt{S_t}dB_tQ+Q^\rmt dB_t^\rmt \sqrt{S_t}+ \al Q^\rmt Q dt,\quad S_0=s_0 \]
and show that
\[ T=\inf\{s:\,\det(S_s)=0\}=\infty \]
where we adopt the idea from \citet[p.734]{bru} to use McKean's argument.
As a matrix norm we choose 
\beqq ||A||:=\max_{k=1:p}\sum_{j=1}^p|A_{jk}| \label{matrixnorm}\eeqq
\vspace{0.5cm}Observe that then $|\tr(A)|\leq p||A||$ for every matrix $A\in\calm_p(\bbr)$.\\
Let us assume that $T<\infty$.\\

First consider the case $\al=p+1$. Then we have from (\ref{eq_4})
\[ d(\ln(\det(S_t)))=2\sqrt{\tr(Q^\rmt Q S_t^{-1})}\,d\beta_t \]
We can define an increasing sequence of stopping times $(T_n)_{n\in\bbn}$ with
\[ T_n:=\inf\{t\in\bbr_+:\,||S_t^{-1}||=n\}, \]
where $T_n<\infty$ because $T<\infty$, that converges to $T$ such that $\ln(\det(S_{\min\{t,T_n\}}))$ is a martingale:
\beao
E\left(\int_0^{T_n}\left(2\sqrt{\tr(Q^\rmt Q S_t^{-1})}\right)^2\,dt\right) & = & 4 E\left(\int_0^{T_n}\tr(Q^\rmt Q S_t^{-1})\,dt\right) \\
& \leq & 4 E\left(\int_0^{T_n} p ||Q^\rmt Q|| n \,dt\right) \\
& = & 4 T_n p ||Q^\rmt Q|| n < \infty
\eeao
where we used that
\beao 0\leq \tr(Q^\rmt Q S_t^{-1})\leq p ||Q^\rmt Q S_t^{-1}||\leq p ||Q^\rmt Q|| \, ||S_t^{-1}|| \leq p ||Q^\rmt Q|| n \eeao
for $t\in [0,T_n)$. That means by definition that $\ln(\det(S_t))$ is a local martingale on $[0,T)$.\\
Now we can apply McKean's argument, Theorem \ref{theorem_mckean}, with $r_t=\det(S_t)$ and $h\equiv ln$. By assumption we know $\det(S_0)>0$ and we have shown above that $h(r_t)=\ln(\det(S_t))$ is a local martingale on $[0,T)$. Obviously, $\ln(\det(S_t))$ converges to -$\infty$ for $t\rightarrow T$ and hence McKean's argument implies $T=\infty$. That is a contradiction to our assumption. Logically consistent, we can conclude $T=\infty$.

In the case $\al>p+1$, we set $\xi=\frac{p+1-\al}{2}<0$ and
\[ T_n:=\inf\{t\in\bbr_+:\,||S_t^{-1}||=n\} \wedge \inf\{t\in\bbr_+:\,\det(S_t^{-1})\geq n\} \]
where $a\wedge b:=\inf\{a,b\}$. Again, $(T_n)_{n\in\bbn}$ is a sequence of stopping times that converges to $T$. From (\ref{eq_2}) we know that
\[ d(\det(S_t)^\xi)=2\xi\det(S_t)^\xi \sqrt{\tr(Q^\rmt Q S_t^{-1})}\,d\beta_t \]
We show again that $\det(S_{t\wedge T_n})^\xi$ is a martingale for every $n\in\bbn$:
\beao
E\left(\int_0^{T_n}\left(2\xi\det(S_t)^\xi\sqrt{\tr(Q^\rmt Q S_t^{-1})}\right)^2\,dt\right)
& = & 4\xi^2 E\left(\int_0^{T_n} \det(S_t)^{2\xi} \tr(Q^\rmt Q S_t^{-1})\,dt\right) \\
& \leq & 4\xi^2 E\left(\int_0^{T_n} n^{-2\xi} p ||Q^\rmt Q|| n \,dt\right) \\
& = & 4\xi^2 T_n p ||Q^\rmt Q|| n^{1-2\xi} < \infty
\eeao
Hence, we can apply McKean's argument for the local martingale $\det(S_t)^\xi$ on $[0,T)$, because $\det(S_t)^\xi$ converges to infinity for $t\rightarrow T$. The same reasoning as above implies the contradiction $T=\infty$.\\
Finally, Theorem \ref{theorem_exwp1} proves the statement.\\
Observe that, for $\al<p+1$, we have $\xi>0$ and thus $\det(S_t)^\xi$ does converge to zero, so we cannot apply McKean's argument in this case.
\end{proof}

Eventually, we are able to state the final theorem about the existence and uniqueness of the Wishart process, which is the main achievement of this section.

\begin{Theorem}[Existence and Uniqueness of the Wishart Process III]\ \\
Let $Q\in GL(p)$, $K\in\calm_p(\bbr)$, $s_0\in\cals_p^+$ and the parameter $\al\geq p+1$. For $\wh{B}\thicksim\calbm_p$ consider the stochastic differential equation
\beqq
d\wh{S_t}=\sqrt{\wh{S_t}}\,d\wh{B_t}\,Q+Q^\rmt d\wh{B_t}^\rmt \sqrt{\wh{S_t}}+ \al\, Q^\rmt Q \,dt ,\quad S_0=s_0
\label{wo_drift}\eeqq
From Theorem \ref{theorem_exwp3} we know that there exists an unique strong solution $(\wh{S},\wh{B})$ of (\ref{wo_drift}) on $[0,T^{\wh{S}})$ with $T^{\wh{S}}=\inf\{s:\det(\wh{S_s})=0\}=\infty$. Define the process
\[ U_t := -\sqrt{\wh{S_t}}KQ^{-1} \]
and suppose that
\beqq \left(\mathcal{E}\left(\tr\left(-\int_0^t U_s^\rmt\,dB_s\right)\right)\right)_{t\in[0,\infty)} \label{girsanov_martingale2}\eeqq
is a martingale .\\
Then there exists an unique strong solution $(S,B)=((S_t)_{t\in\bbr_+},(B_t)_{t\in\bbr_+})$, $B\thicksim\calbm_p$, in the cone of all positive definite matrices $\cals_p^+$ of the Wishart SDE
\beqq
dS_t=\sqrt{S_t}\,dB_t\,Q+Q^\rmt dB_t^\rmt \sqrt{S_t}+(S_t K+K^\rmt S_t+ \al \, Q^\rmt Q)\,dt,\quad S_0=s_0
\label{wishartsde}\eeqq
on the entire interval $[0,\infty)$.
\label{theorem_exwp}
\end{Theorem}

\begin{proof}
In this proof, with solution we always mean an $\cals_p^+$-valued solution.\\

With Girsanov's Theorem (Theorem \ref{girsanov}) we are able to conclude that
\beqq B_t := \int_0^t U_s\,dt + \wh{B_t}=-\int_0^t \sqrt{\wh{S_s}}KQ^{-1} \,dt + \wh{B_t} \label{eq_bhat}\eeqq
defines a Brownian motion w.r.t. the equivalentprobability measure $\wh{Q}$ as defined in (\ref{eq_q}).\\

Some calculation shows that
\beam
d\wh{S_t} & = & \sqrt{\wh{S_t}}\,d\wh{B_t}\,Q+Q^\rmt d\wh{B_t}^\rmt \sqrt{\wh{S_t}}+ \al Q^\rmt Q \,dt \nonumber\\
& = &  \sqrt{\wh{S_t}}\,(dB_t+\sqrt{\wh{S_t}}KQ^{-1}\,dt)\,Q + Q^\rmt d\wh{B_t}^\rmt \sqrt{\wh{S_t}}+ \al Q^\rmt Q \,dt \nonumber\\
& = &  \sqrt{\wh{S_t}}\,dB_t\,Q + Q^\rmt d\wh{B_t}^\rmt \sqrt{\wh{S_t}}+ \wh{S_t}K\,dt + \al Q^\rmt Q \,dt \nonumber\\
& = &  \sqrt{\wh{S_t}}\,dB_t\,Q+Q^\rmt dB_t^\rmt \sqrt{\wh{S_t}}+(\wh{S_t} K+K^\rmt \wh{S_t}+ \al Q^\rmt Q)\,dt 
\label{wishartsde2}\eeam
Hence $(\wh{S},B)$ is a solution of (\ref{wishartsde}) on $[0,\infty)$.\\

From Theorem \ref{theorem_exwp1} we know that there exists a unique strong solution $S$ of (\ref{wishartsde2}) on the interval $[0,T^S)$ with $T^S=\inf\{s:\,\det(S_s)=0\}>0$. 
The pathwise uniqueness implies uniqueness in law, thus $S$ and $\wh{S}$ have the same distribution, and so have $T^S$ and $T^{\wh{S}}$. Hence, $T^S=T^{\wh{S}}=\infty$ and $(S,B)$ is an unique strong solution $S$ of (\ref{wishartsde}) on the interval $[0,\infty)$.
\end{proof}

\begin{Remark}
In the case that the matrices $Q^\rmt Q$ and $K$ commute, \citet[p. 748]{bru} has shown that (\ref{girsanov_martingale2}) is a martingale by extending the methods of \citet{pitman_jor} for the one-dimensional case. 
\end{Remark}

\begin{Assumption}
For the rest of this thesis, we will assume that (\ref{girsanov_martingale2}) is a martingale.
\end{Assumption}

Before we end this chapter, we want to compare our results to the one stated in \citet{bru}:

\begin{Theorem}(Cf. \citet[Theorem 2'']{bru})
If $\al\in \Delta_p := \{1,\ldots,p-1\}\cup (p-1,+ \infty)$, $A\in GL(p)$, $B\in \cals_p^{-}$, $s_0\in\cals_p^{+}$ and has all its Eigenvalues distinct, and $M$ is a $p\times p$-dimensional Brownian motion, then the stochastic differential equation
\beqq
dS_t = S_t^{\frac{1}{2}}\,dM_t\,(A^\rmt A)^{\frac{1}{2}} + (A^\rmt A)^{\frac{1}{2}}\,dM_t^\rmt\,S_t^{\frac{1}{2}} + (BS_t + S_tB)\,dt + \al A^\rmt A\,dt, S_0=s_0
\label{sde_bru}\eeqq
has a unique solution on $[0,\tau)$ if $B$ and $(A^\rmt A)^{\frac{1}{2}}$ commute, whereas $\tau$ denotes the first time of collision, i.e. the first time that two eigenvalues of $S$ become equal. With the term unique solution is meant a weak solution that is unique in law.
\label{exwp_bru}\end{Theorem}

Compared to (\ref{sde_bru}), our definition of the Wishart SDE (\ref{eq_wp}) is more general, as we allow the drift matrix $K$ to be an arbitrary matrix, whereas in (\ref{sde_bru}) the matrix $B$ has to be symmetric negative definite.
In the case $\al\in[p+1,\infty)$, the result in Theorem \ref{theorem_exwp} extends the one in Theorem \ref{exwp_bru} because we prove the existence of a strong solution, that is unique up to indistinguishability and takes almost surely values in the cone of all symmetric positive definite matrices. Furthermore, this solution has infinite lifetime independently of any collision of the eigenvalues of our solution.

\section{Square Ornstein-Uhlenbeck Processes and their Distributions}

\begin{Theorem}\ \\
Let $n\in\{p+1,p+2,\ldots\}$, $A\in GL(p)$, $B\in\calm_p(\bbr)$, $s_0\in\cals_p^+$ and $X\thicksim\caloup_{n,p}(A,B,x_0)$ be an Ornstein-Uhlenbeck process with $x_0^\rmt x_0=s_0$. Then there exists a Brownian motion $M\thicksim\calbm_p$ such that the unique strong solution $(S,M)$ in $\cals_p^+$ of the stochastic differential equation
\beqq
dS_t = S_t^{\frac{1}{2}}\,dM_t\,(A^\rmt A)^{\frac{1}{2}} + (A^\rmt A)^{\frac{1}{2}}\,dM_t^\rmt\,S_t^{\frac{1}{2}} + (B^\rmt S_t + S_tB)\,dt + n A^\rmt A\,dt, \quad S_0=s_0
\label{eq_soup}\eeqq
is given by the square Ornstein-Uhlenbeck process $S=X^\rmt X=(X_t^\rmt X_t)_{t\in\bbr_+}$.
\label{soup}\end{Theorem}

Note that the class of stochastic differential equations of the form (\ref{eq_soup}) is exactly the class of Wishart SDEs (\ref{eq_wp}) with $Q=\sqrt{A^\rmt A}\in\cals_p^+$, $K=B\in\calm_p(\bbr)$ and\\
$\al=n\in\{p+1,p+2,\ldots\}$. Thus, everything we have established in the foregoing section is still valid for the subclass of SDEs (\ref{eq_soup}).

\begin{proof}
We define
\[ S_t := X_t^\rmt X_t \quad \fa t\in\bbr_+ \]
and
\[ M_t := \int_0^t \sqrt{S_s^{-1}}X_s^\rmt\,dW_s\,A(\sqrt{A^\rmt A})^{-1} \in\calm_p(\bbr)\quad \fa t\in\bbr_+ \]
The matrix square root $\sqrt{A^\rmt A}$ is positive definite and therefore invertible. Now we show that $M$ is a Brownian motion. Because of
\beao
& & E\left( \int_0^t (\sqrt{S_s^{-1}}X_s^\rmt A(A^\rmt A)^{-\frac{1}{2}})^\rmt (\sqrt{S_s^{-1}}X_s^\rmt A(A^\rmt A)^{-\frac{1}{2}})\,ds \right) \\
&= & E\left( \int_0^t (A^\rmt A)^{-\frac{1}{2}} A^\rmt X_s S_s^{-1} X_s^\rmt A (A^\rmt A)^{-\frac{1}{2}} \,ds \right) \\
&= & \int_0^t (A^\rmt A)^{-\frac{1}{2}} A^\rmt A (A^\rmt A)^{-\frac{1}{2}} \,ds \\
&= & t I_p < \infty \textnormal{ a.s.}
\eeao
$M_t$ is a local martingale. Furthermore, observe that
\[ dM_{t,ij} = \sum_{m,n} (\sqrt{S_t^{-1}}X_t)_{im} \,dW_{t,mn}\, (A(\sqrt{A^\rmt A})^{-1})_{nj} \]
and
\beao
d[M_{ij},M_{kl}]_t & = & \sum_{m,n} (\sqrt{S_t^{-1}}X_t)_{im} (\sqrt{S_t^{-1}}X_t)_{km} (A(\sqrt{A^\rmt A})^{-1})_{nj} (A(\sqrt{A^\rmt A})^{-1})_{nl} \,dt \\
& = & (\sqrt{S_t^{-1}}X_t X_t^\rmt \sqrt{S_t^{-1}})_{ik} ( (\sqrt{A^\rmt A})^{-1} A^\rmt A (\sqrt{A^\rmt A})^{-1} )_{jl} \\
& = & (I_p)_{ik} (I_p)_{jl} \,dt \\
& = & \mathbf{1}_{\{i=k\}} \mathbf{1}_{\{i=l\}}  \,dt
\eeao
where we used that
\[ d[W_{mn},W_{m',n'}]_t=dt \Leftrightarrow m=m',\, n=n' \textnormal{ , and zero otherwise} \]
With Theorem \ref{theorem_levy} we con conclude that $M$ is a Brownian motion.\\
Finally, using the partial integration formula shows us that our pair $(S,M)$ is a solution of the stochastic differential equation (\ref{eq_soup}).
\beao
dS_t & = & d(X_t^\rmt X_t) = (dX_t)^\rmt\,X_t + X_t^\rmt\,(dX_t) + d[X^\rmt,X]_t^M \\
& = & (B^\rmt X_t^\rmt\,dt + A^\rmt\,dW_t^\rmt)\,X_t + X_t^\rmt\,(X_tB\,dt + dW_t\,A) + A^\rmt\,d[W_t^\rmt,W_t]_t^M\,A \\
& = & B^\rmt S_t\,dt + A^\rmt\,dW_t\,X_t + S_tB\,dt + X_t^\rmt\,dW_t\,A + A^\rmt\,d[W_t^\rmt,W_t]_t^M\,A\\
& = & (B^\rmt S_t + S_tB)\,dt + A^\rmt\,dW_t\,X_t + X_t^\rmt\,dW_t\,A + A^\rmt\,d[n I_p t]^M A \\
& = & (B^\rmt S_t + S_tB)\,dt + (A^\rmt A)^{\frac{1}{2}}\,dM_t^\rmt\,S_t^{\frac{1}{2}} + S_t^{\frac{1}{2}}\,dM_t\,(A^\rmt A)^{\frac{1}{2}} + nA^\rmt A\,dt \\
& = & S_t^{\frac{1}{2}}\,dM_t\,(A^\rmt A)^{\frac{1}{2}} + (A^\rmt A)^{\frac{1}{2}}\,dM_t^\rmt\,S_t^{\frac{1}{2}} + (B^\rmt S_t + S_tB)\,dt + nA^\rmt A\,dt
\eeao
where we used that
\[ d[W_t^\rmt,W_t]_{t,ij}^M= \sum_{k=1}^n d[W_{t,ki},W_{t,kj}]_t = n \mathbf{1}_{\{i=j\}} dt \]
So far we have shown that $(S,M)$ is a weak solution of (\ref{eq_soup}). From Theorem \ref{theorem_exwp} we know that pathwise uniqueness holds for (\ref{eq_soup}), and thus by Remark \ref{remark_pu} (ii) we have that $(S,M)$ is also a strong solution of (\ref{eq_soup}).
\end{proof}

\begin{Theorem}[Conditional Distribution of the Square Ornstein-Uhlenbeck Process]\ \\
Let $n\in\{p+1,p+2,\ldots\}$, $A\in GL(p)$, $s_0\in\cals_p^+$ and $B\in\calm_p(\bbr)$ with $0\notin-\sigma(B)-\sigma(B)$. The solution of (\ref{eq_soup}) has the conditional distribution 
\beqq S_t|s_0\thicksim\calw_p(n,\Sigma_t,\Sigma_t^{-1}e^{B^\rmt t} s_0 e^{Bt}) \label{cd_soup}\eeqq
with
\beqq \Sigma_t= \cala^{-1}(A^\rmt A)-\cala^{-1}(e^{B^\rmt t} A^\rmt A e^{Bt} ) \eeqq
where $\cala^{-1}$ is the inverse of $\cala:\cals_p\rightarrow\cals_p,X\mapsto-B^\rmt X-XB$.
\label{theorem_dswp}\end{Theorem}

\begin{proof}
The solution $S=X^\rmt X$ is given by a square Ornstein-Uhlenbeck process.
Theorem \ref{theorem_doup} shows that
\[ X_t|x_0 \thicksim\caln_{n,p}(x_0e^{Bt},I_n \otimes \Sigma_t ) \]
with
\[ \Sigma_t= \cala^{-1}(A^\rmt A)-\cala^{-1}(e^{B^\rmt t} A^\rmt A e^{Bt} ) \]
From Lemma \ref{lemma_normaltranspose} we know
\[ X_t^\rmt |x_0 \thicksim\caln_{p,n}(e^{B^\rmt t}x_0^\rmt,\Sigma_t \otimes I_n ) \]
Using Lemma \ref{lemma_wd} we achieve
\[ X_t^\rmt X_t |x_0 = X_t^\rmt (X_t^\rmt)^\rmt |x_0 \thicksim\calw_p(n,\Sigma_t,\Sigma_t^{-1}e^{B^\rmt t}x_0^\rmt x_0e^{Bt}) \]
i.e.
\[ S_t|s_0\thicksim\calw_p(n,\Sigma_t,\Sigma_t^{-1}e^{B^\rmt t} s_0 e^{Bt}) \]
\end{proof}

\begin{Theorem}[Stationary Distribution of the Square Ornstein-Uhlenbeck Process]\ \\
Let $n\in\{p+1,p+2,\ldots\}$, $A\in GL(p)$, $s_0\in\cals_p^+$ and $B\in\calm_p(\bbr)$ with $Re(\sigma(B))\subseteq (-\infty,0)$. Then the solution of (\ref{eq_soup}) has a stationary limiting distribution, that is
\beqq
\calw_p(n,\cala^{-1}(A^\rmt A),0)
\eeqq
where $\cala^{-1}$ is the inverse of $\cala:\cals_p\rightarrow\cals_p,X\mapsto-B^\rmt X-XB$.
\end{Theorem}

\begin{proof}
From (\ref{cd_soup}) and Theorem \ref{cfwishart} we know for the solution $S$ of (\ref{eq_soup}) that $S_t$ given $S_0$ has characteristic function
\beqq \wh{P^{S_t}}=\det(I_p-2i\Sigma_t Z)^{-\frac{n}{2}} etr[i\Theta_t(I_p-2i\Sigma_t Z)^{-1}\Sigma_t Z] \eeqq
with
\[ \Sigma_t= \cala^{-1}(A^\rmt A)-\cala^{-1}(e^{B^\rmt t} A^\rmt A e^{Bt} ) \]
and
\[ \Theta_t= \Sigma_t^{-1}e^{B^\rmt t} s_0 e^{Bt} \]
From the proof of Theorem \ref{statd_oup} we know that
\[ Re(\sigma(B))\subseteq (-\infty,0) \Rightarrow \lim_{\tto} \exp(Bt)=0 \]
and 
\[ \lim_{\tto} \Sigma_t=\cala^{-1}(A^\rmt A) \]
Thus
\[ \lim_{\tto} \Theta_t=0 \]
and
\[ \lim_{\tto} \wh{P^{S_t}}=\det(I_p-2i\, \cala^{-1}(A^\rmt A)\, Z)^{-\frac{n}{2}}=:f(Z) \]
With L\'evy's Continuity Theorem it exist a probability measure $\mu$ such that $\wh{\mu}(Z)=f(Z)$ for all $Z\in\calm_p(\bbr)$ and $P^{S_t}\limw\mu$. According to Remark \ref{centralwishart} and Theorem \ref{cfwishart}, $f$ is the characteristic function of a random matrix with central Wishart distribution $\calw_p(n,\cala^{-1}(A^\rmt A),0)$.
Hence, $S$ has limit distribution $\calw_p(n,\cala^{-1}(A^\rmt A),0)$ regardless of any initial value $s_0$. As $S$ is a Markov process (Theorem \ref{markovsolution}) we conclude that $\calw_p(n,\cala^{-1}(A^\rmt A),0)$ is its stationary distribution.
\end{proof}

In the case where $B\in\cals_p^-$, and the matrices $A^\rmt A$ and $B$ commute, it holds that $\cala^{-1}(A^\rmt A)=-\frac{1}{2}A^\rmt A B^{-1}$ and the stationary limiting distribution is $\calw_p(n,-\frac{1}{2}A^\rmt A B^{-1},0)$.

\newpage
\section{Simulation of Wishart Processes}

\label{simwis}

Recall the stochastic differential equation of the Wishart process, that is
\[ dS_s=\sqrt{S_s}\,dB_s\,Q+Q^\rmt\,dB_s^\rmt\,\sqrt{S_s}+(S_s K+K^\rmt S_s+\al Q^\rmt Q)\,ds,\quad S_0=s_0 \]
For every $t\in\bbr_+,\,h>0$ integration over the interval $[t,t+h]$ yields
\beqq S_{t+h}=S_t+\int_t^{t+h} \sqrt{S_s}\,dB_s\,Q + \int_t^{t+h} Q^\rmt\,dB_s^\rmt\,\sqrt{S_s} + \int_t^{t+h} S_s K\,ds + \int_t^{t+h} K^\rmt S_s \,ds + \al Q^\rmt Q h \eeqq
We now try to approximate the stochastic integral above to make the Wishart process suitable for numerical simulation. An easy way of doing this, is the Euler-Maruyama method (see \citet[p. 340]{kloeden}):
\beqq \wh{S}_{t+h}=\wh{S}_t+ \sqrt{\wh{S}_t}(B_{t+h}-B_t)Q + Q^\rmt (B_{t+h}^\rmt-B_t^\rmt)\sqrt{\wh{S}_t} + ( \wh{S}_t K +K^\rmt \wh{S}_t+\al Q^\rmt Q )h \label{disWP}\eeqq
We call $\wh{S}$ the discretized Wishart process. As the Brownian motion has stationary, independent increments, we
know that the distribution of $B_{t+h}-B_t$ is $\caln_p(0,h I_{2p})$ and is independent of all previous increments. Hence, we can use (\ref{disWP}) to simulate $\wh{S}$ for any fixed step size $h>0$. Then we get a process on the mesh $(0,h,2h,\ldots,T)$ for any $T\in\bbr_+$, that is $(\wh{S}_0,\wh{S}_h,\wh{S}_{2h},\ldots,\wh{S}_T)$.\\
However, even under the Assumptions of Theorem \ref{theorem_exwp}, this discretized Wishart process can become negative definite such that we have to stop our simulation before we reach $T$, because the square root in (\ref{disWP}) is not well-defined anymore.\\
To solve this , we observe that $\wh{S}\rightarrow S$ for $h\rightarrow 0$. Thus, we can expect $\wh{S}$ to remain positive semidefinite as long as $h$ is sufficiently small. Hence, we introduce a variable step size to our algorithm:\\
Suppose we have already given the discretized Wishart process $(\wh{S}_0,\wh{S}_h,\wh{S}_{2h},\ldots,\wh{S}_t)$ and the discretized Brownian motion $(0,B_h,B_{2h},\ldots,B_t)$ up to a time $t<T-h$. Suppose further that we calculate $\wh{S}_{t+h}$ (and thus $B_{t+h}$) according to (\ref{disWP}) and that $\wh{S}_{t+h}$ is negative definite, i.e. (at least) its smallest eigenvalue becomes negative. Then, we cut our step size by half, calculate $\wh{S}_{t+\frac{h}{2}}$ and check again if the smallest eigenvalue of $\wh{S}_{t+\frac{h}{2}}$ is negative. We continue this iteratively, until (hopefully) $\wh{S}_{t+\frac{h}{2^n}}$ is positive semidefinite or, the step size falls under a certain value, say $\approx 2.2\times 10^{-16}$ (that is the constant \emph{eps} in MATLAB). In the last case, the step size converges to zero.\\
In order to get the value $\wh{S}_{t+\frac{h}{2}}$, we need to draw $B_{t+\frac{h}{2}}$ conditionally on the given values $(0,B_h,B_{2h},\ldots,B_t,B_{t+h})$. Because of the Markov property of the Brownian motion (see Theorem \ref{markovsolution}), this is the same as drawing $B_{t+\frac{h}{2}}$ conditionally on $(B_t,B_{t+h})$. For now, we only consider the $i,j-th$ entry $B_{ij}$ of $B$. From \citet[p.84]{glasserman} we know that
\beao B_{ij,t+\frac{h}{2}}| (B_{ij,t}=x,B_{ij,t+h}=x+y) & \eqd & \frac{\frac{h}{2}x+\frac{h}{2}(x+y)}{h}+\sqrt{\frac{\frac{h}{2}\frac{h}{2}}{h}}Z_{ij} \\ 
& = & x+\frac{1}{2}y+\frac{\sqrt{h}}{2}Z_{ij} 
\eeao
where $Z_{ij}\thicksim\caln_1(0,1)$ and $\eqd$ denotes distributional equivalence. Written in matrix notation, we get for the increment of $B$
\beqq (B_{t+\frac{h}{2}}-B_t)| (B_{t}=X,B_{t+h}=X+Y) \eqd \frac{1}{2}Y+\frac{\sqrt{h}}{2}Z \eeqq
where $Z\thicksim\caln_{p,p}(0,I_{2p})$. A sample implementation can be found in the appendix. Now we give a few examples.

\begin{Example}
First we begin with the one-dimensional square Bessel process, i.e. a solution of
\[ dX_t=2\sqrt{X_t}\,d\beta_t+\al\,dt, \quad X_0=x_0:=0.5 \]
where $\beta$ denotes a one-dimensional Brownian motion. From Theorem \ref{theorem_expupbesq} we know that this process never becomes negative for any choice of $\al,x_0\geq 0$ and remains positive for $x_0>0$ and $\al\geq 2$. Here are two sample paths over the interval $[0,1]$ and initial step size $h=10^{-3}$:
\includegraphics{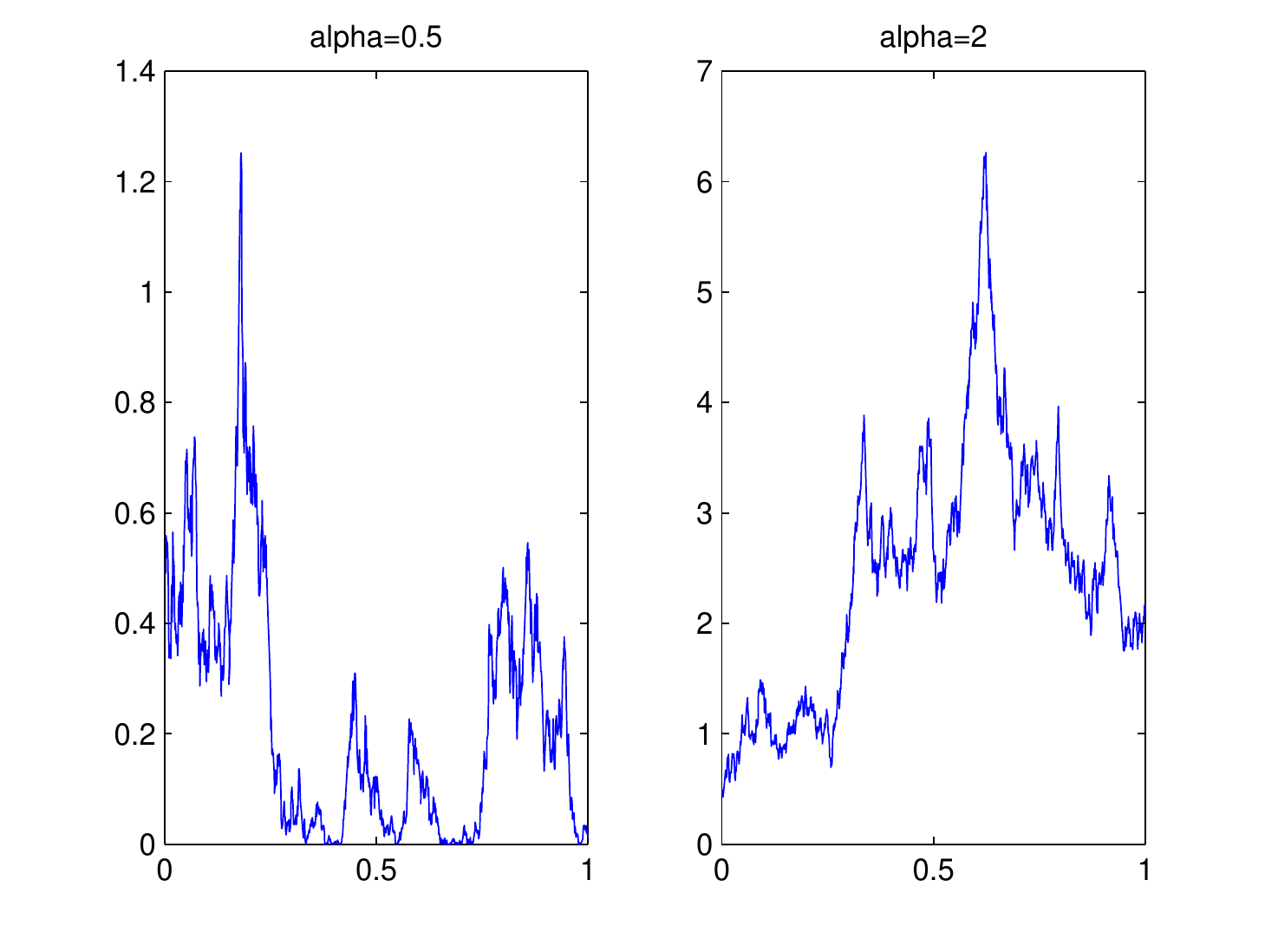}\\
In the case $\alpha=0.5$, the algorithm had to reduce the initial step size to $1.5625\cdot10^{-5}$ at at least one point in order to guarantee non-negativity, whereas for $\alpha=2$ this was not necessary. Observe that in the case $\alpha=0.5<2$, the square Bessel process can become zero, but gets reflected instantaneously.
\end{Example}

\begin{Example} 
Our next example is the 2-dimensional Wishart process
\[ dS_t=\sqrt{S_t}\,dB_t\,Q+Q^\rmt dB_t^\rmt \sqrt{S_t}+(S_t K+K^\rmt S_t+ \al Q^\rmt Q)\,dt,\quad S_0=s_0 \]
with a 2-dimensional Brownian motion $B$, $\al=3$, $Q=(\begin{smallmatrix} 1 & 2 \\ 0 & -3 \end{smallmatrix})$, $K=-4I_2$ and $s_0=I_2$. We show sample paths for the three different entries of $S$, that are $S_{11}$, $S_{22}$ and $S_{12}$, and for $\frac{S_{12}}{S_{11}^2+S_{22}^2}$, what would be the correlation in a stochastic volatility model.\\
\includegraphics{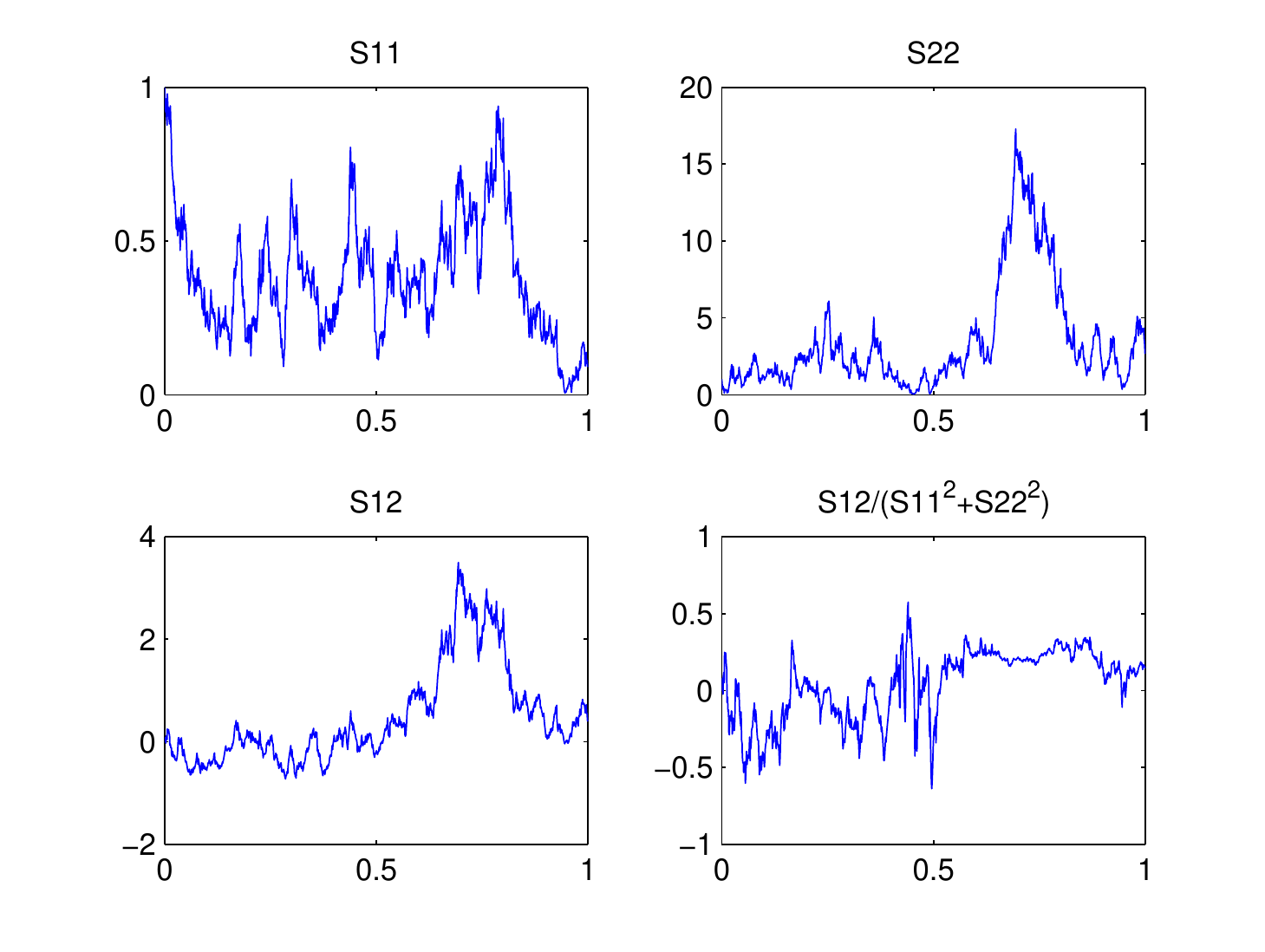}\\
The algorithm had to reduce the initial step size of $h=10^{-3}$ by half at some points. We have\\
\includegraphics[scale=0.8]{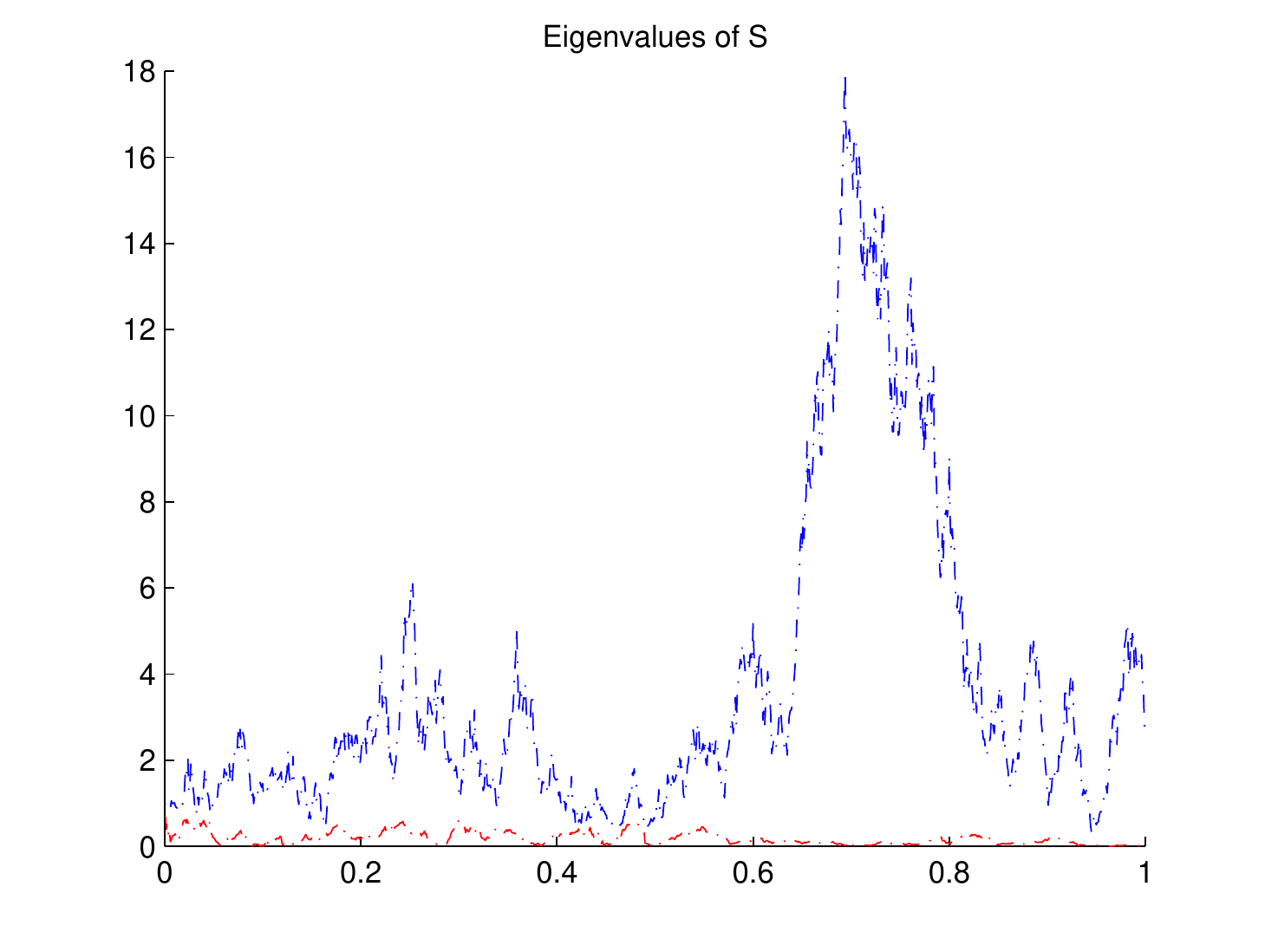}\\
for the eigenvalues of $S$.
\end{Example}

\begin{Example}
The last example are two 5-dimensional solutions of the stochastic differential equation
\beqq dS_t=\sqrt{S_t}\,dB_t+ dB_t^\rmt \sqrt{S_t} + (-4 S_t + \al_i )\,dt,\quad S_0=0.1\cdot I_5, i=1,2 \label{wp_example3}\eeqq
For the first one, we have $\al_1=7.2$, and the smallest and the largest eigenvalue of $S$ look like\\
\includegraphics{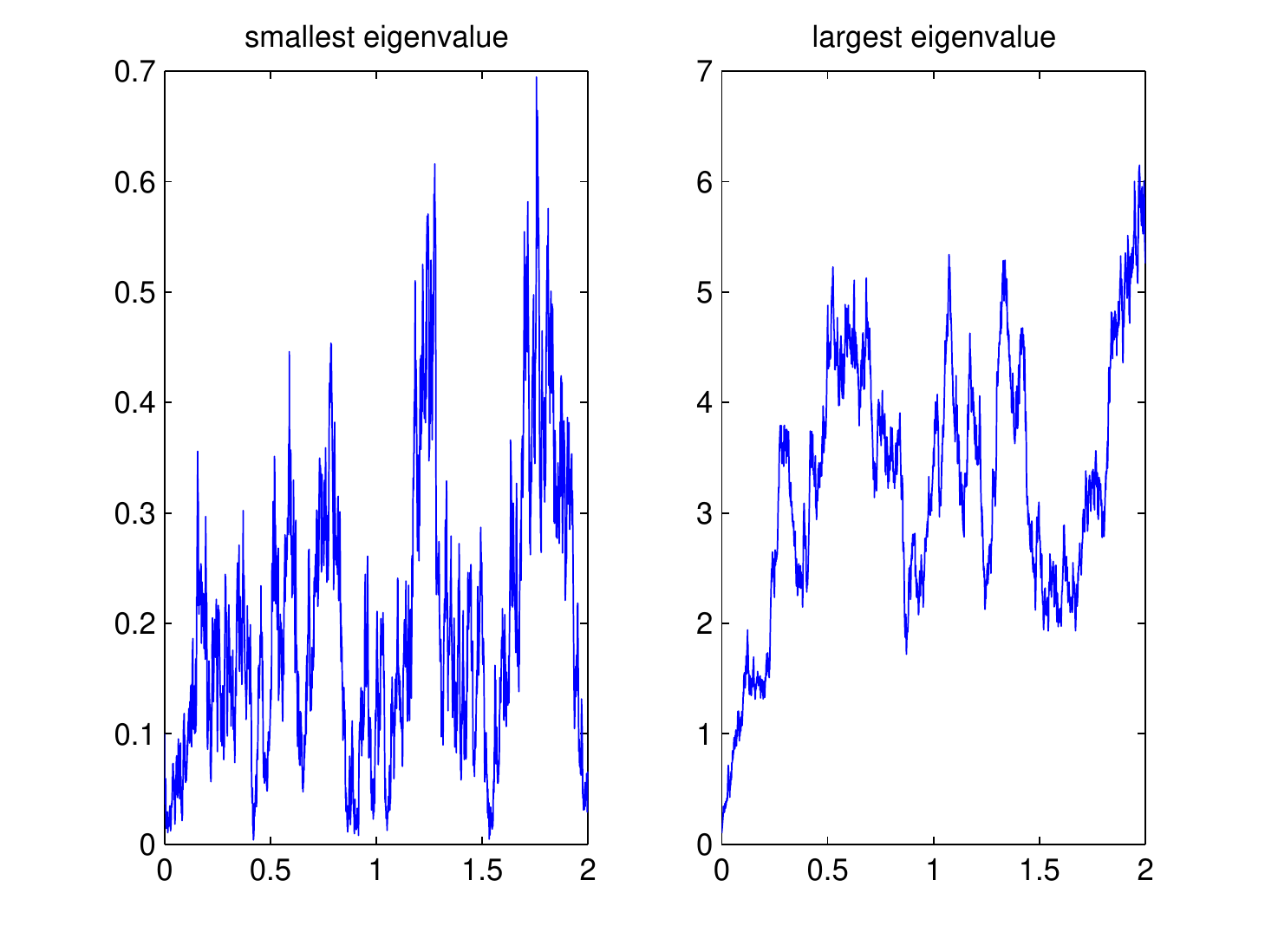}\\
In the second case, $\al_2=3.5$, we don't know if there exists any solution of (\ref{wp_example3}) for $i=2$. Thus, we do not know whether our simulated process corresponds to a solution of (\ref{wp_example3}) for $i=2$, and if it does, it does not have to be a Wishart process by Definition \ref{def_wp} (because we demand the Wishart process to be a strong solution). It may be noted, that the algorithm had to reduce the initial step size of $h=10^{-3}$ to $7.8125\cdot 10^{-6}$ at some points. Again, we shall have a look at the eigenvalues of our simulated process:\\
\includegraphics{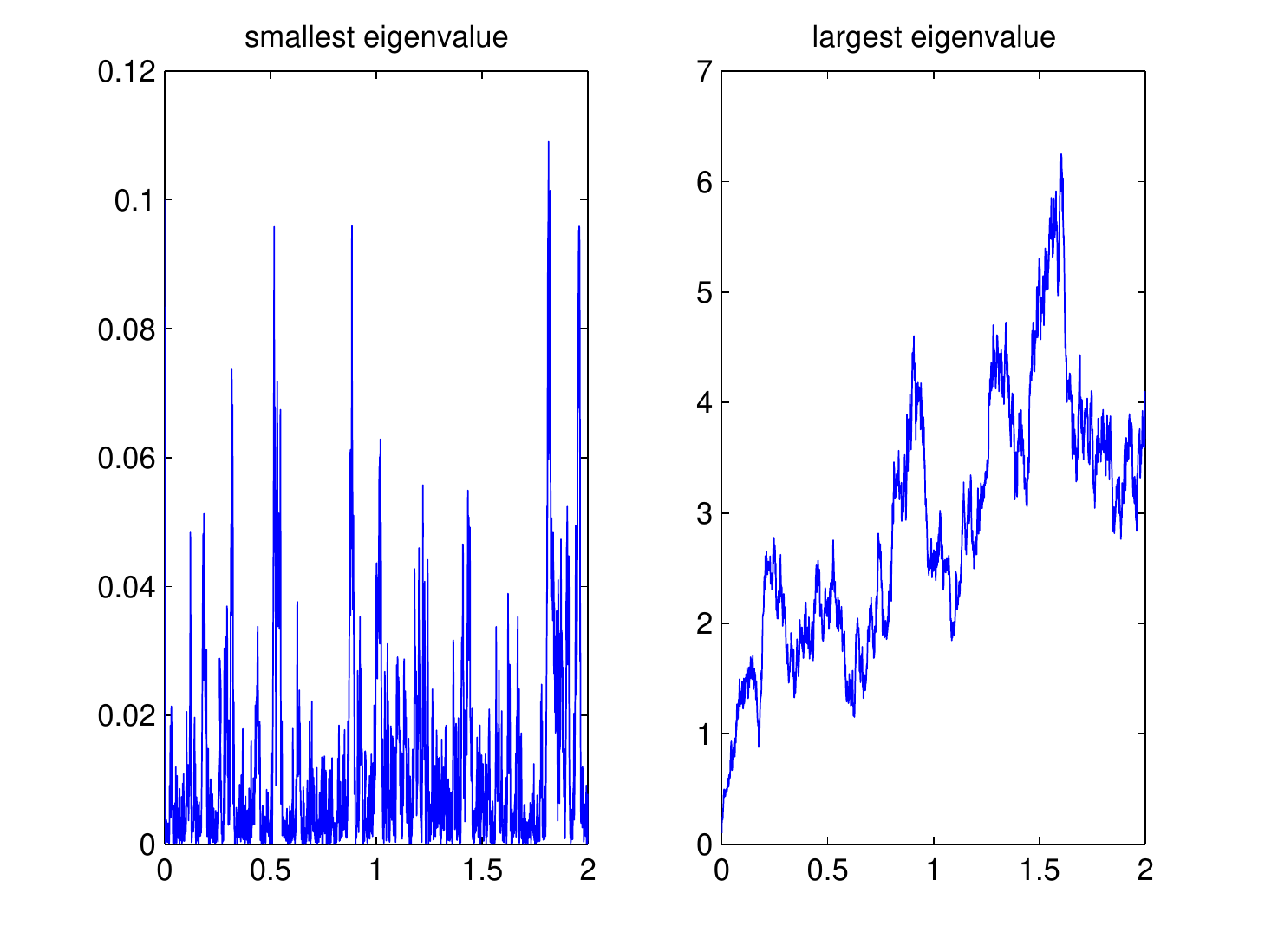}
\end{Example}

\chapter{Financial Applications}

We first focus on two well-known financial models that are based on the one-dimensional Wishart processes, that is in fact a generalized squared Bessel process. Later, we consider the actual matrix variate case.

\section{CIR Model}\label{cir}

According to Theorem \ref{theorem_exwp}, we get for $p=1$, $q> 0$, $k\in\bbr$, $\al\geq p+1=2$, $s_0>0$ and an one-dimensional Brownian motion $B$ that the stochastic differential equation
\[ dS_t=2q\sqrt{S_t}\,dB_t+(2kS_t+\al q^2)\,dt,\quad S_0=s_0 \]
has a unique, strong and positive solution in $(0,\infty)$. For $k<0$ and the new parametrization $\sigma=2q$, $a=-2k$ and $b=\frac{-\al q^2}{2k}$ the stochastic differential equation gets the form
\beqq dS_t=\sigma\sqrt{S_t}\,dB_t+a(b-S_t)\,dt \label{eq_cir}\eeqq
that is the stochastic differential equation of the CIR process. It is mean reverting as $a>0$ with  $b\geq\frac{-q^2}{k}$.\\
As \citet[p.122]{glasserman} has shown, $S_t|s_0$ is distributed as $\frac{\sigma^2(1-\exp(-at))}{4a}$ times a\\ noncentral chi-square random variable with $\frac{4ab}{\sigma^2}$ degrees of freedom and noncentrality parameter $\frac{4a\exp(-at)}{\sigma^2(1-\exp(-at))}s_0$.\\
From now on we follow \citet[p.184f.]{gou}. Let's assume that the interest rate $r$ follows a stochastic differential equation of the form (\ref{eq_cir}),
\beqq dr_t=\sigma\sqrt{r_t}\,dW_t+a(b-r_t)\,dt \label{eq_cir2}\eeqq
where $W$ is a Brownian motion under a risk-neutral probability measure $Q$. Then, the price of a zero-coupon bond at time $t$ with time to maturity $h$ is given by
\beqq B(t,t+h)=E_t^Q\left[\exp\left(-\int_t^{t+h}r_{\tau}\,d\tau\right)\right] \label{eq_zcb}\eeqq
where $E_t^Q$ denotes the conditional expectation $E^Q[\cdot|\sigma\{r_s:\,s\leq t\}]$ under the measure $Q$.
Then \citet{cox} have shown that
\[ B(t,t+h)=\exp(-f(h)r_t-g(h)] \]
with functions
\beao
f(h) & = & \frac{2}{c+a}-\frac{4c}{(c+a)[(c+a)\exp(ch)+c-a]} \\
g(h) & = & -\frac{ab(c+a)h}{\sigma^2}+\frac{2ab}{\sigma^2}\ln\left(\frac{(c+a)\exp(ch)+c-a}{2c}\right)
\eeao
with $c:=\sqrt{a^2+2\sigma^2}$.\\
Hence, we have a closed form solution for $B(t,t+h)$ that is exponential affine in $r_t$.

\section{Heston Model}
A process $S$ of the form (\ref{eq_cir}) can also be used to model the volatility in a stochastic volatility Black-Scholes model according to Heston
\[ d(\ln(X_t))=\mu_t\,dt+\sqrt{S_t}\,dW_t \]
where the stock price at time $t$ is denoted by $X_t$. We refer to \citet[Chapter 2.2.2.]{gou} for details.

\section{Factor Model for Bonds}
In contrast to section \vref{cir} where the risk-free rate $r$ followed a stochastic differential equation of the form (\ref{eq_cir2}), we now want to use a factor model to consider corporate bonds jointly with a long term government bond (e.g. T-bond). Again, we summarize the results of \citet[chapter 3.5.2.]{gou}.\\
Denote by $\lambda_{i,t}$ the default intensity for firm $i$, $i=1,\ldots,K$.
Let us assume that
\beam
r_t & = & c+\tr(CS_t) \label{eq_r}\\
\lambda_{i,t} & = & d_i+\tr(D_iS_t) \label{eq_lambda}\quad\fa i=1,\ldots,K
\eeam
where $c,d_i\geq 0$ are nonnegative, $C,D_i\in\cals_p^+$ and $S\thicksim\calwp_p(Q,K,\al,s_0)$, i.e.
\beqq dS_t=\sqrt{S_t}dB_tQ+Q^\rmt dB_t^\rmt \sqrt{S_t}+(KS_t+S_tK^\rmt+ \al Q^\rmt Q)dt,\quad S_0=s_0 \label{finance_wishart}\eeqq
under the conditions of Theorem \ref{theorem_exwp}.\\
Equation (\ref{eq_r}) is the factor representation for the risk-free rate and equation (\ref{eq_lambda}) the factor representation for the corporate bonds.\\
Observe, that $\tr(CS_t)>0$ for all $t$. Indeed, as $C\in\cals_p^+$ there exists an orthogonal matrix such that $C=UDU^\rmt$ where $D$ is a diagonal matrix of eigenvalues $\mu_k\geq 0$. Denote by $u_i\in\bbr^p$, $i=1,\ldots,p$, the columns of $U$. Then we have $C=\sum_{k=1}^p\mu_k u_k u_k^\rmt$ and thus
\[ \tr(CS_t)=\sum_{k=1}^p\mu_k\tr(u_k u_k^\rmt S_t)=\sum_{k=1}^p\mu_k u_k^\rmt S_t u_k \geq 0 \]
With the convention $\lambda_{0,t}=0$ we get for the price $B_0$ of a zero-coupon bond and the prices of corporate bonds $B_i$ the formula
\beqq B_i(t,t+h)=E_t^Q\left[\exp\left(-\int_t^{t+h}r_{\tau}+\lambda_{i,\tau}\,d\tau\right)\right] \quad\fa i=0,\ldots,K \label{eq_zcb2}\eeqq
under a risk neutral measure $Q$.
If we insert (\ref{eq_r}) and (\ref{eq_lambda}) into (\ref{eq_zcb2}) we get
\beqq B_i(t,t+h)=\exp(-h(c+d_i)) E_t^Q\left[\etr\left(-(C+D_i)\int_t^{t+h}S_{\tau}\,d\tau\right)\right] \quad\fa i=0,\ldots,K \label{eq_bond}\eeqq
with the convention $d_0,D_0\equiv0$. \citet{gou} shows that there exists a closed form expression for (\ref{eq_bond}).

\section{Matrix Variate Stochastic Volatility Models}
An application for matrix variate stochastic processes can be found in \citet{fonseca}, which model the dynamics of $p$ risky assets by
\beqq dX_t=\diag(X_t)[(r\mathbf{1}+\lambda_t)\,dt+\sqrt{S_t}\,dW_t] \eeqq
where $r$ is a positive number, $\mathbf{1}=(1,\ldots,1)\in\bbr^p$, $\lambda_t$ a $p$-dimensional stochastic process, interpreted as the risk premium, and $Z$ a $p$-dimensional Brownian motion. The volatility process $S$ is the Wishart process of (\ref{finance_wishart}).\\

In contrast to a continous model of the form (\ref{eq_fonseca}), another multivariate stochastic volatility model for the logarithmic stock price process can be given by
\beqq dY_t=(\mu+\Sigma_t\beta)\,dt+\Sigma_t^{\frac{1}{2}}\,dW_t,\quad Y_0=0 \eeqq
where $\mu,\beta\in\bbr^p$, $W$ denotes a $p$-dimensional Brownian motion and $\Sigma$ is given by a Lévy-driven positive semidefinite OU type process, see \cite{stelzer} for details. In this case, the volatility is not continuous anymore and has jumps.\\

If one wants to model the volatility with a time continuous stochastic process, e.g. for economic reasons, one could also suggest the model
\beqq dY_t=(\mu+S_t\beta)\,dt+S_t^{\frac{1}{2}}\,dW_t,\quad Y_0=0 \eeqq
where the process $\Sigma$ was substituted by the Wishart process of (\ref{finance_wishart}).

\appendix

\chapter{MATLAB Code for Simulating a Wishart Process}
This is an implementation in MATLAB of the algorithm described in section \vref{simwis}.

\begin{verbatim}
function [S,eigS,timestep,minh]=Wishart(T,h,p,alpha,Q,K,s_0)
%
% Simulates the p-dimensional Wishart process on the interval [0,T] 
% that follows the stochastic differential equation 
%   dS_t=sqrt{S_t}*dB_t*Q+Q'*dB_t'*sqrt{S_t}+(S_t*K+K'*S_t+alpha*Q'*Q)dt
% with initial condition S_0=s_0.
%
% Method of discretization: Euler-Maruyama
% Starting step size: h
% In order to guarantee positive semidefiniteness of S, the step size will be
% reduced iteratively if necessary.
%
% Output:
% S is a three dimensional array of the discretized Wishart process. 
% eigS is a matrix consisting the eigenvalues of S
% timestep is the vector of all timesteps in [0,T]
% minh is the smallest step size used, i.e. minh=min(diff(timestep))
%
% Author: Oliver Pfaffel
% Email: O.Pfaffel@gmx.de
% June 17, 2008
%
%--------------------------------------------------------------------

horg=h;
minh=h;
timestep=0;

[V_0,D_0]=eig(s_0);
eigS=sort(diag(D_0)');
drift_fix=alpha*Q'*Q;

B_old=0;
B_inc=sqrt(h)*normrnd(0,1,p,p); 
vola=V_0*sqrt(D_0)*V_0'*B_inc*Q;
drift=s_0*K;
S_new = s_0+vola+vola'+(drift+drift'+drift_fix)*h;

[V_new,D_new]=eig(S_new);
eigS=[eigS;sort(diag(D_new)')];

S=cat(3,s_0,S_new);
t=h;
timestep=[timestep;t];

flag=0;

while t+h<T,
    
    B_old=B_old+B_inc;
    B_inc=sqrt(h)*normrnd(0,1,p,p);
    
    S_t=S_new; V_t=V_new; D_t=D_new;
    
    sqrtm_S_t=V_t*sqrt(D_t)*V_t';
    vola=sqrtm_S_t*B_inc*Q;
    drift=S_t*K;
    S_new = S_t+vola+vola'+(drift+drift'+drift_fix)*h;
    
    [V_new,D_new]=eig(S_new);
    
    mineig=min(diag(D_new));
    
    while mineig<0,
        
        h=h/2;
        minh=min(minh,h);
        flag=1;
        
        if h<eps, error('Step size converges to zero'), return, end
        
        B_inc=0.5*B_inc+sqrt(h/2)*normrnd(0,1,p,p);

        vola=sqrtm_S_t*B_inc*Q;
        drift=S_t*K;
        S_new = S_t+vola+vola'+(drift+drift'+drift_fix)*h;

        mineig=min(eig(S_new));
        
    end
    
    if flag==0, 
        eigS=[eigS;sort(diag(D_new)')];
        S=cat(3,S,S_new);
        t=t+h;
        timestep=[timestep;t];
    end
    
    if flag==1,
        [V_new,D_new]=eig(S_new);
        eigS=[eigS;sort(diag(D_new)')];
        S=cat(3,S,S_new);
        t=t+h;
        timestep=[timestep;t];
        flag=0;
        h=horg;
    end      
    
end

h_end=T-t;

if h_end>0,
    
    B_inc=sqrt(h_end)*normrnd(0,1,p,p);
    
    S_t=S_new; V_t=V_new; D_t=D_new;

    vola=V_t*sqrt(D_t)*V_t'*B_inc*Q;
    drift=S_t*K;
    S_new = S_t+vola+vola'+(drift+drift'+drift_fix)*h;
    
    S=cat(3,S,S_new);
    eigS=[eigS;sort(eig(S_new)')];
    timestep=[timestep;T];
    
end

\end{verbatim}


\begin{thebibliography}{27}
\providecommand{\natexlab}[1]{#1}
\providecommand{\url}[1]{\texttt{#1}}
\expandafter\ifx\csname urlstyle\endcsname\relax
  \providecommand{\doi}[1]{doi: #1}\else
  \providecommand{\doi}{doi: \begingroup \urlstyle{rm}\Url}\fi

\bibitem[Barndorff-Nielsen and Stelzer(2007)]{barndorff}
Ole~Eiler Barndorff-Nielsen and Robert Stelzer.
\newblock Positive-definite matrix processes of finite variation.
\newblock \emph{Probability and Mathematical Statistics}, 27:\penalty0 3--43,
  2007.

\bibitem[Bru(1991)]{bru}
Marie-France Bru.
\newblock Wishart processes.
\newblock \emph{Journal of Theoretical Probability}, 4:\penalty0 725 -- 751,
  1991.

\bibitem[B.v.Querenburg(2001)]{bvq}
B.v.Querenburg.
\newblock \emph{Mengentheoretische Topologie}.
\newblock Springer, 2001.

\bibitem[Cox et~al.(1985)Cox, Ingersoll, and Ross]{cox}
J.~Cox, J.~Ingersoll, and S.~Ross.
\newblock A theory of the term structure of interest rates.
\newblock \emph{Econometrica}, 53:\penalty0 385--407, 1985.

\bibitem[Deuflhard and Hohmann(2002)]{num}
P.~Deuflhard and A.~Hohmann.
\newblock \emph{Numerische Mathematik I}.
\newblock de Gruyter Lehrbuch, 2002.

\bibitem[Fischer(2005)]{fischer}
Gerd Fischer.
\newblock \emph{Lineare Algebra}.
\newblock Vieweg, 2005.

\bibitem[Fonseca et~al.(2008)Fonseca, Grasselli, and Tebaldi]{fonseca}
J.~Da Fonseca, M.~Grasselli, and C.~Tebaldi.
\newblock Option pricing when correlations are stochastic: an analytical
  framework.
\newblock \emph{Springer}, 2008.

\bibitem[Glasserman(2004)]{glasserman}
Paul Glasserman.
\newblock \emph{Monte Carlo Methods in Financial Engineering}.
\newblock Springer-Verlag, 2004.

\bibitem[Gourieroux(2007)]{gou}
C.~Gourieroux.
\newblock Continuous time {W}ishart process for stochastic risk.
\newblock \emph{Econometric Reviews}, 25:2:\penalty0 177 -- 217, 2007.

\bibitem[Gupta and Nagar(2000)]{grupta}
A.~K. Gupta and D.~K. Nagar.
\newblock \emph{Matrix variate distributions}.
\newblock Chapman \& Hall/CRC, 2000.

\bibitem[Jacod and Protter(2004)]{pe}
J.~Jacod and P.~Protter.
\newblock \emph{Probability Essentials}.
\newblock Springer-Verlag, 2004.

\bibitem[Jurek and Mason(1993)]{mason}
Zbigniew~J. Jurek and J.~David Mason.
\newblock \emph{Operator-Limit Distributions in Probability Theory}.
\newblock Wiley Series in Probability and Mathematical Statistics, 1993.

\bibitem[Kallenberg(1997)]{kalle}
Olav Kallenberg.
\newblock \emph{Foundations of Modern Probability}.
\newblock Springer-Verlag, 1997.

\bibitem[Kloeden and Platen(1999)]{kloeden}
Peter~E. Kloeden and Eckhard Platen.
\newblock \emph{Numerical Solution of Stochastic Differential Equations}.
\newblock Springer, 1999.

\bibitem[McKean(1969)]{mckean}
H.~P. McKean.
\newblock \emph{Stochastic Integrals}.
\newblock Academic Press, 1969.

\bibitem[Métivier and Pellaumail(1980b)]{met}
M.~Métivier and J.~Pellaumail.
\newblock \emph{Stochastic Integration}.
\newblock Academic Press, 1980b.

\bibitem[Muirhead(2005)]{head}
R.~J. Muirhead.
\newblock \emph{Aspects of Multivariate Statistical Theory}.
\newblock Wiley, 2005.

\bibitem[{\O}ksendal(2000)]{oek}
Bernt {\O}ksendal.
\newblock \emph{Stochastic Differential Equations: An Introduction with
  Applications}.
\newblock Springer-Verlag, 2000.

\bibitem[Olkin and Rubin(1961)]{olkin}
Ingram Olkin and Herman Rubin.
\newblock A characterization of the {W}ishart distribution.
\newblock \emph{The Annals of Mathematical Statistics}, pages 1272--1280, 1961.

\bibitem[Pitman and Yor(1982)]{pitman_jor}
Jim Pitman and Marc Yor.
\newblock A decomposition of {B}essel bridges.
\newblock \emph{Z. Wahrscheinlichkeitstheorie verw. Gebiete}, 59:\penalty0
  425--457, 1982.

\bibitem[Protter(2004)]{protter}
Philip~E. Protter.
\newblock \emph{Stochastic Integration and Differential Equations}.
\newblock Springer-Verlag Berlin Heidelberg, 2004.

\bibitem[Revuz and Yor(2001)]{revuz}
Daniel Revuz and Marc Yor.
\newblock \emph{Continuous Martingales and Brownian Motion}.
\newblock Springer-Verlag, 2001.

\bibitem[Skorohod(1965)]{sko}
A.~V. Skorohod.
\newblock \emph{Studies in the theory of random processes}.
\newblock Addison-Wesley, 1965.

\bibitem[Stelzer(2007)]{stelzer}
Robert Stelzer.
\newblock \emph{Multivariate Continuous Time Stochastic Volatility Models
  Driven by a L\'evy Process}.
\newblock PhD thesis, Centre for Mathematical Sciences, Munich University of
  Technology, 2007.

\bibitem[Timmann(2005)]{rep}
Steffen Timmann.
\newblock \emph{Repetitorium der Gewöhnlichen Differentialgleichungen}.
\newblock Binomi, 2005.

\bibitem[Yamada and Watanabe(1971{\natexlab{a}})]{yamada1}
T.~Yamada and S.~Watanabe.
\newblock On the uniqueness of solutions of stochastic differential equations.
\newblock \emph{J. Math. Kyoto Univ.}, 11-1:\penalty0 155--167,
  1971{\natexlab{a}}.

\bibitem[Yamada and Watanabe(1971{\natexlab{b}})]{yamada2}
T.~Yamada and S.~Watanabe.
\newblock On the uniqueness of solutions of stochastic differential equations
  {II}.
\newblock \emph{J. Math. Kyoto Univ.}, 11-3:\penalty0 553--563,
  1971{\natexlab{b}}.

\end{thebibliography}

\end{document}